\documentclass[11pt]{article}
\usepackage{latexsym,amsfonts,amssymb,mathrsfs,float}
\usepackage{dsfont}
\usepackage{amsthm}
\usepackage[numbers]{natbib}
\usepackage{amsmath}
\usepackage{indentfirst}
\usepackage{bbm}

\usepackage{ulem}
\usepackage{color}
\usepackage{colordvi,multicol}
\allowdisplaybreaks

\usepackage{hyperref}
\hypersetup{colorlinks=true, linkcolor=blue}

\newtheorem{thm}{Theorem}[section]

\newtheorem{lem}[thm]{Lemma}
\newtheorem{pro}[thm]{Proposition}

\newtheorem{rmk}[thm]{Remark}

\numberwithin{equation}{section}

\begin{document}

\title{\bf Favorite sites of one-dimensional transient random walk}
 \author{ Zechun Hu$^{1}$,
 Renming  Song$^{2}$, Guangshuo Zhou$^{1,}$\footnote{Corresponding author}\ ,   and Qianqian Zhou$^{3}$\\ \\
  {\small $^1$ College of Mathematics, Sichuan  University,
 Chengdu, 610065, China}\\
 {\small zchu@scu.edu.cn, zhgsh3363@163.com}\\ \\
  {\small $^2$ Department of Mathematics,
University of Illinois Urbana-Champaign, Urbana, IL 61801, USA}\\
 {\small rsong@illinois.edu} \\  \\
 {\small $^3$ School of Mathematical Sciences,
Tianjin University of Technology,  Tianjin, 300384, China}\\
 {\small qianqzhou@yeah.net}}

 \date{}
\maketitle

\noindent{\bf Abstract:}\quad  In this paper, we study favorite sites of one-dimensional transient random  walk on $\mathbb{Z}$ beyond the simple case.
Denote by $\mathcal{K}(n)$ the set of favorite sites at time $n$.  The goal of this paper is twofold. First, we study $\# \mathcal{K}(n)$, the cardinality of $\mathcal{K}(n)$, and show that,  under some mild conditions, for any positive integer $k$, $\mathbf{P}(\#\mathcal{K}(n)=k, {\rm i.o.})=1$, and $\limsup_{n\to\infty}\#\mathcal{K}(n)/\log\log n=-1/\log \gamma$, where $\gamma$ is the probability that the random walk never returns to the starting point. Second, we study the escape rate of favorite sites, and give the limsup and liminf behaviors, where the liminf behavior is characterized via an integral test.  To the best of our knowledge,
this integral test is the first result of this nature for random walks (in non-random environments).

\smallskip


\smallskip


\section{Introduction and main results}\setcounter{equation}{0}

Let $\{S_n\}_{n\ge0}$ be a random walk on the integer lattice $\mathbb{Z}$,
defined by $S_0=x$  and $S_n=x+\sum_{i=1}^{n} X_i$ for $n\ge1$, where
$\{X_i\}_{i\ge 1}$ is a sequence of independent and identically distributed
(i.i.d.) integer-valued random variables.
We denote by $\mathbf{P}_x$ and $\mathbf{E}_x$ the law of the random walk starting from $x$ and the corresponding expectation, and write $\mathbf{P}$ for $\mathbf{P}_0$ and $\mathbf{E}$ for $\mathbf{E}_0$ for simplicity.

For any  $z\in \mathbb{Z}$ and integer  $n\ge1$, the local time  of the walk at
$z$ up to time $n$ is defined as
$$
\xi(z,n):= \#\{0\leq k\leq n: S_k=z\}.
$$
Hereafter, $\# A$ denotes the cardinality of the set $A$.
Let $\xi^*(n):=\max_{z\in\mathbb Z}\xi(z,n)$ denote the maximal local time at time
$n$.
The set of favorite sites (or most visited sites) is defined by
$$
\mathcal{K}(n):=\{z\in\mathbb{Z}:\xi(z,n)=\xi^*(n)\},\quad  n\geq 1.
$$
We will use $F_n$ to denote an arbitrary element in $\mathcal{K}(n)$.

Before giving our motivation and the main results of this paper, we first recall some results
on the favorite sites (edges) of simple symmetric random walks on  $\mathbb{Z}^d$ for $d\geq 1$. For convenience, we use $\mathcal{K}^{(d)}(n)$ to denote the set of favorite sites at time $n$ in this case.

The cardinality of $\mathcal{K}^{(d)}(n)$ has been intensively studied. It is easy to see that, for $k=1,2$, $\#\mathcal{K}^{(d)}(n)=k$ occurs infinitely often (i.o.) almost surely (a.s.).
Erd\H{o}s and R\'{e}v\'{e}sz \cite{ER1984,ER1987} posed the following question:
\begin{align}\label{1.1}
 {\rm Can}\ \#\mathcal{K}^{(d)}(n)=k\ \mbox{occur i.o. for}\ k\geq 3\ {\rm and}\ d\geq 1?
\end{align}
For $d\geq 3$, the answer is yes: Erd\H{o}s and R\'{e}v\'{e}sz \cite{ER1991}  proved that $\mathbf{P}(\#\mathcal{K}^{(d)}(n)=k\ {\rm i.o.})=1$ for any $k\geq 1$. When $d=1$, T\'oth and  Werner \cite{TW1997} proved that there are no more than three favorite edges eventually, which means that $\mathbf{P}(\#\mathcal{K}^{(1)}_e(n)\geq 4\ {\rm i.o.})=0$, where $\mathcal{K}^{(1)}_e(n)$ denotes the set of favorite edges at time $n$. T\'oth \cite{T2001} showed that a.s. $\#\mathcal{K}^{(1)}(n)\leq 3$ eventually.
In 2018, Ding and Shen \cite{DS2018} proved that a.s. $\#\mathcal{K}^{(1)}(n)=3$ i.o..
Hao et al. \cite{HHMS2026} proved that a.s. ``three favorite edges" occurs i.o..
Hao et al. \cite{HLOZ2025} showed that a.s. $\limsup_{n\to\infty}\#\mathcal{K}^{(2)}(n)=3$, giving a complete answer to the question \eqref{1.1}.
In \cite{HLOZ2025}, it was also proved that for $d\geq 3$,
\begin{align}\label{1.2}
\limsup_{n\to\infty}\frac{\#\mathcal{K}^{(d)}(n)}{\log\log n}=-\frac{1}{\log\gamma_d} ~~\text{a.s.},
\end{align}
where $\gamma_d:=\mathbf{P}(S_n\neq S_0,\forall n\geq 1)$, the probability that a simple random walk never returns to the starting point.  Zhou et al. \cite{ZHS2025} studied one-dimensional asymmetric simple random walks, and proved that their behavior is similar to that of $d$-dimensional simple symmetric random walks with $d\geq 3$, i.e. $\mathbf{P}(\#\mathcal{K}(n)=k\ {\rm i.o.})=1$ for any $k\geq 1$ and a similar result to (\ref{1.2}) holds.

The asymptotic behavior of favorite sites has also been intensively studied.
For the upper limits of $|F_n|$, Erd\H{o}s and R\'{e}v\'{e}sz \cite{ER1984} and Bass and Griffin \cite{BG1985} independently proved the following:
$$
\limsup_{n\to\infty}\frac{|F_n|}{(2n\log\log n)^{1/2}}=1~~\text{a.s.}.
$$
For the lower limits of $|F_n|$, the first result was proved by Bass and Griffin \cite{BG1985}, which says that
\begin{eqnarray}\label{1.3}
\liminf_{n\to\infty}\frac{|F_n|}{n^{1/2}(\log n)^{-\kappa}}=
\left\{
\begin{array}{ll}
0 & \mbox{if}~~ \kappa<1,\\
\infty & \mbox{if}~~\kappa>11,
\end{array}
\right.\quad \text{a.s.}.
\end{eqnarray}
A natural question was

\quad\quad {\it Does there exist $\kappa_0$ such that $\liminf\limits_{n\to\infty}\frac{|F_n|}{n^{1/2}(\log n)^{-\kappa}}$ is 0 if $\kappa<\kappa_0$ and $\infty$ if $\kappa>\kappa_0$?}

\noindent
There is a similar problem for Brownian motion. \eqref{1.3} implies that, if such a $\kappa_0$ exists, then $\kappa_0\in [1,11]$. Lifshits and Shi \cite{LS2004} showed that the lim inf is 0 when $\kappa=1$. Since then, several papers tried to shrink the
interval $[1,11]$.  Bass et al. \cite{BES2000} proved that $\kappa_0\leq9$.
Marcus and Rosen \cite{MR2006} showed that $\kappa_0\leq 3$. Bass \cite{B2023} completely solved this question and showed that $\kappa_0=1$.  Hu and Shi \cite{HS2000} studied favorite sites of Sinai's simple random walk in random environment, proved that the favorite site process is transient,  and characterized the escape rate of the favorite sites via an integral criterion. To the best of our knowledge, no such integral test was known for the escape rate of the favorite sites of random walks (in non-random environments) or Brownian motion. We refer the reader to Shi and T\'{o}th \cite{ST2000} and Okada \cite{Oka16} for two surveys on favorite sites.

Up to now, all the existing results on favorite sites of random walks are for simple random walks.
We will study one-dimensional transient random walks beyond the simple case.
We will study the cardinality of $\mathcal{K}(n)$ and the asymptotic behavior of favorite sites.

Throughout this paper, we assume that $\{S_n\}_{n\geq 0}$ is a transient random walk on $\mathbb{Z}$ with $S_0=0$ and $\mathbf{E}[|X_1|]<\infty$. Without loss of generality, we assume that $ \mu:=\mathbf{E}[X_1]\in (0,\infty)$. By the strong law of large numbers, $S_n\to+\infty$ almost surely. We also always assume the following condition:
\begin{equation}\label{eq:1.1}
\mathbf{P}(X_1<0)>0.
\end{equation}
This condition excludes the trivial case that the random walk has only positive jumps.
Without loss of generality, we assume that the random walk is defined on the canonical path space.

Now we are ready to state the main results of this paper. Theorems 1.1 and 1.2 are on the cardinality of $\mathcal{K}(n)$. Theorems 1.3 and 1.4 are on the asymptotic behavior of favorite sites.
\begin{thm}\label{thm:1.2} For every integer $k\ge1$,
\begin{align}\label{thm:1.2-a}
\mathbf P\left(\#\mathcal{K}(n)=k\ {\rm i.o.}\right)=1.
\end{align}
\end{thm}

Define
\begin{equation}\label{eq:1.2}
\gamma:=\mathbf{P}(S_n\neq0,\ \forall n\ge1).
\end{equation}
Since $\mu\in (0, \infty)$ and \eqref{eq:1.1} holds, we know that $\gamma\in (0, 1)$.

\begin{thm}\label{thm:1.3}
If  $\mathbf{E}[X_1^2]<\infty$, then almost surely,
$$
\limsup_{n\to\infty}
\frac{\#\mathcal{K}(n)}{\log\log n}
=
-\frac{1}{\log\gamma}.
$$
\end{thm}

\begin{thm}\label{thm:1.4}
For any $F_n\in \mathcal{K}(n)$,
$$
\limsup_{n\to\infty}\frac{F_n}{\mu n}=1\ \textup{a.s.}.
$$
\end{thm}

\begin{thm}\label{thm:1.5}
For any $F_n\in \mathcal{K}(n)$ and non-decreasing $a: \mathbb{N}\to (0, \infty)$,
$$
\liminf_{n\to\infty}\frac{a(n)}{n}F_n=
\begin{cases}
0,\\
\infty,
\end{cases}
\textup{a.s.}
~\Longleftrightarrow~
\sum_{n=2}^{\infty}\frac{1}{n\,a(n)}
\begin{cases}
=\infty,\\
<\infty.
\end{cases}
$$
 In particular,
 for   $\kappa>0$,  almost surely,
\begin{eqnarray*}
\liminf_{n\to\infty}\frac{(\log n)^{\kappa}}{n}F_n=
\left\{
\begin{array}{ll}
0& \mbox{if} \quad
\kappa\leq 1,\\
\infty&  \mbox{if}
\quad  \kappa>1.
\end{array}
\right.
\end{eqnarray*}
\end{thm}

To prove Theorem~\ref{thm:1.2}, we first construct successively $k$ distinct favorite sites \and then use transience together with Kolmogorov's 0-1 law to obtain that almost surely ``$k$ favorite sites"  occurs  infinitely often. The proof of Theorem~\ref{thm:1.3} follows the general idea of the proof of the corresponding result in Hao et al. \cite{HLOZ2025}, with substantial modifications. The key ingredient is the two-point local time estimate in Lemma \ref{pro:2.3}  below, which controls the occurrence of two competing thick points within $O((\log n)^2)$ steps. Combined with the eventual ordering of successive favorite sites and the Borel-Cantelli lemma, this yields the precise $\log\log n$ asymptotic behavior under $\mathbf{E}[X_1^2]<\infty$.

Results similar to Theorems~\ref{thm:1.4} and~\ref{thm:1.5} have been established for Brownian motion with drift in \cite{BM1998}. Transferring these results via the strong invariance principle for local times generally requires at least moment assumptions of order $2+\varepsilon$.
So we will use a direct method which only uses $\mathbf{E}[|X_1|]<\infty$ and $\mathbf{P}(X_1<0)>0$.

The proofs of Theorems~\ref{thm:1.4} and~\ref{thm:1.5} are based on regeneration epochs induced by the strict ascending ladder epochs that are never subsequently undershot. By decomposing the path at the regeneration epochs and using spatial homogeneity, Proposition~\ref{pro:3.3} proves that the corresponding regeneration block vectors are i.i.d..
It also implies the strong laws of large numbers for the regeneration times and heights. Lemma~\ref{lem:3.4} further shows that the collections of local times accumulated within the regeneration blocks are i.i.d..
 The favorite sites can therefore be localized through the first and last block indices $I_r^-$ and $I_r^+$ attaining the maximal block local time.
The i.i.d. structure of the block maxima and bounds on the number of
maximizers yield sharp scale estimates for these indices. These estimates connect the location of the favorite sites with the record structure of the block local time maxima. Theorem~\ref{thm:1.4} follows from the occurrence of arbitrarily large block records and the linear growth of the regeneration heights. To prove Theorem~\ref{thm:1.5}, we first give an  integral test for $I_r^\pm$ in the regeneration block scale, and then use Proposition~\ref{pro:3.3} to transfer the integral test to the original time and spatial scales.

For $a,b\in\mathbf{R}$, write $a\wedge b:=\min\{a,b\}$ and $a\vee b:=\max\{a,b\}$.
For a set $A$, we use  $1_A$ to denote its indicator function. For $x\in\mathbf{R}$, $\lfloor x\rfloor$ stands for the greatest integer not exceeding $x$,  $\lceil x\rceil$ stands for the smallest integer which is larger than or equal to $x$. Define
$$
\mathcal{F}_n:=\sigma(S_0,S_1,\ldots,S_n),\ \forall n\geq 0,
~~\text{and}~~
\mathcal{F}_\infty:=\sigma\left(\cup_{n\geq0}\mathcal{F}_n\right).
$$
For integers $0\leq i<j$,
we define
$
S_{(i,j)}:=(S_{i+1},\ldots,S_{j-1}),
~
S_{[i,j]}:=(S_i,\ldots,S_j),
~
S_{[i,j)}:=(S_i,\ldots,S_{j-1}).
$
For a transition kernel $Q=(Q(x,y))_{x,y\in E}$ on a countable state space $E$,
$Q^{(k)}$ stands for its $k$-fold convolution, i.e., $Q^{(0)}(x,y):=1_{\{x=y\}}$ and, for $k\geq1$,
$$
Q^{(k)}(x,y):=\sum_{x_1,\ldots,x_{k-1}\in E}Q(x,x_1)Q(x_1,x_2)\cdots Q(x_{k-1},y).
$$
For a matrix $B=(B(x,y))_{x,y\in E}$ indexed by a finite set $E$, we define
$$
\|B\|_\infty:=\max_{x\in E}\sum_{y\in E}|B(x,y)|.
$$
Throughout the paper, $c$ , $C$ and $c_i$ denote generic positive constants whose values may change from one appearance  to another.

The organization of this paper is as follows.
The proofs of Theorems \ref{thm:1.2} and \ref{thm:1.3} are given in Sections 2.1 and 2.2, respectively, and the proofs of three lemmas, Lemmas 2.1, 2.4 and 2.6, used to prove Theorems \ref{thm:1.2} and \ref{thm:1.3}, are given in Section 2.3.
In Sections 3.1 and 3.2 we establish properties of the regeneration blocks and their local times. In Sections 3.3 and 3.4, we prove Theorems \ref{thm:1.4} and \ref{thm:1.5}, respectively.  Several auxiliary lemmas used in the proof of Theorem \ref{thm:1.5} are proved in Section 3.5.

\section{Proofs of Theorems \ref{thm:1.2} and  \ref{thm:1.3}}\label{sec:2}\setcounter{equation}{0}

For $z\in\mathbb{Z}$, let $\hat{\tau}_z$ be the first hitting time of
$z$:
$$
\hat{\tau}_z:=\inf\{n\ge1:~S_n=z\}.
$$
Then for the $\gamma$  defined in (\ref{eq:1.2}), we have
$\gamma=\mathbf{P}(\hat{\tau}_0=\infty).$
Let
\begin{equation}\label{eq:2.1}
\gamma_n:=\mathbf{P}(\hat{\tau}_0>n),
\end{equation}
be the probability that the random walk does not return to the starting point in the first $n$ steps.

We now introduce some stopping times used in the sequel.  For $m,k\ge1$, define
\begin{align*}
&T_m^0:=0,\quad T_1^1:=0,\\
&T_m^k:=\inf\left\{n>T_m^{k-1}:\#\{z\in\mathbb{Z}:\xi(z,n)\ge m\}=k\right\}\quad \mbox{for any}\ (m,k)\neq (1,1),
\end{align*}
and
$$
\mathcal{F}_m^k:=\mathcal{F}_{T_m^k}:=
\left\{A\in\mathcal{F}_{\infty}: A\cap\{T_m^k\leq n\}\in\mathcal{F}_n \text{ for all }n\geq0 \right\}.
$$
Define $L_m^k:=S_{T_m^k}$ on $\{T_m^k<\infty\}$ and $L_m^k:=+\infty$ otherwise. Then $L_m^k$ is $\mathcal{F}_m^k$-measurable and is the $k$-th site whose local time
is at least $m$
on $\{T_m^k<\infty\}$.

For $m\ge1$, we define
$$
A_m^1:=\Omega,~~\widetilde{A}_m^1:=\Omega,
$$
and
$$
A_m^2:=\left\{S_n\ne L_m^1~\text{for all}~n\in(T_m^1,T_m^2\wedge T_{m+1}^1]\right\},
~~
\widetilde{A}_m^2:=\Omega.
$$
For $k\geq3$, we define
$$
\begin{aligned}
A_m^k
:=
&\left\{S_n\notin\{L_m^1,\ldots,L_m^{k-2}\}~\text{for all}~ n\in\left[T_m^{k-1}+
 \frac{m^2}{16},
T_m^k\wedge T_{m+1}^1\right]\right\}\\
&~\cap\left\{S_n\ne L_m^{k-1}~\text{for all}~n\in(T_m^{k-1},T_m^k\wedge T_{m+1}^1]\right\},
\end{aligned}
$$
and
$$
\widetilde{A}_m^k
:=
\left\{S_n\notin\{L_m^1,\ldots,L_m^{k-2}\}~\text{for all}~ n\in\left[T_m^{k-1}, \left(T_m^{k-1}+
\frac{m^2}{16}
\right)  \wedge T_m^k\wedge T_{m+1}^1\right]\right\}.
$$
For $k\geq 1$ and $m\geq 1$, we set
\begin{equation}\label{eq:2.2}
B_m^k:=\bigcap_{j=1}^k A_m^j,
~~
\widetilde{B}_m^k:=\bigcap_{j=1}^k \widetilde{A}_m^j.
\end{equation}

For $k\geq 1$ and $m\geq 2$, define
\begin{equation}\label{eq:2.3}
C_m^k:=\left\{\exists n\geq1:\#\mathcal{K}(n)=k
\text{ and }\xi^*(n)=m\right\}.
\end{equation}
Since before time $T_{m+1}^1$ no site has local time larger than $m$ and  at time $T_m^k$ exactly $k$ sites have local time at least $m$, it is easy to see that
\begin{equation}\label{eq:2.3-a}
C_m^k=\{T_m^k<T_{m+1}^1\}.
\end{equation}
If $\xi^*(n)\to\infty$ almost surely, we have that for every $k\ge1$,
$$
\left\{\#\mathcal{K}(n)\ge k\text{ infinitely often in}\ n\right\}
\stackrel{\mathrm{a.s.}}{=}
\left\{C_m^k\text{ infinitely often in}\ m\right\}.
$$
It is easy to see from the definitions that $C_m^k\in\mathcal{F}_{m+1}^1$ and
\begin{equation}\label{eq:2.4}
C_m^k=B_m^k\cap\widetilde{B}_m^k.
\end{equation}

\subsection{Proof of Theorem \ref{thm:1.2}}

To prove Theorem \ref{thm:1.2}, we need the following two lemmas.

\begin{lem}\label{lem:2.1}
It holds that
  \begin{equation}\label{eq:2.5}
    \mathbf{P} \left\{\lim_{n\to \infty} \frac{\xi^*(n)}{\log n} =\lambda
    \right\}=1,
  \end{equation}
where  $\lambda:=-1/\log (1-\gamma)$  with the convention that $\lambda=0$ if $\gamma=1$.
\end{lem}
The proof of Lemma \ref{lem:2.1} is postponed to Section 2.3. Using
$(A_n~ \text{\rm i.o.})=\cap_{n=1}^{\infty}\cup_{k=n}^{\infty}A_k$, we can easily get the following result.

\medskip

\begin{lem}\label{lem:2.2}
Let $\left\{A_n\right\}_{n \geq 1}$ be a sequence of events and $c \in(0,1]$. If
$\mathbf{P}\left(A_n\right) \geq c$ for all $n\ge 1$,  then
$\mathbf{P}\left(A_n~ \text{\rm i.o.} \right) \geq c .$
\end{lem}

\begin{proof}[\bf Proof of Theorem~\ref{thm:1.2}]
Fix $k\geq 1$.
Define $\gamma_+ := \mathbf{P}(S_n>0,\ \forall n\geq1)$. \cite[Lemma~2.1]{Vilkas2025} yields $\gamma_+>0$. For $m\geq1$, let $\widetilde{C}_m^1:=\Omega,$ and for $2\leq j\leq k$, define
$$
\widetilde{C}_m^j :=\left\{S_n>S_{T_m^{j-1}}\ \mbox{for all}\ n \in \left(T_m^{j-1}, T_m^j\right]\right\}.
$$
For $2\leq j\leq k$, the event
$\bigcap_{i\geq1}\left\{S_{T_m^{j-1}+i}>S_{T_m^{j-1}}\right\}$ is contained in $\widetilde{C}_m^j$.
Hence the strong Markov property at $T_m^{j-1}$ and spatial homogeneity give
$$
\mathbf{P}\left(\widetilde{C}_m^j\mid\mathcal{F}_m^{j-1}\right)\geq \gamma_+, ~~\text{a.s..}
$$
Since
$\bigcap_{i=1}^{j-1}\widetilde{C}_m^i \in\mathcal{F}_m^{j-1},$ successive conditioning gives, for $k\geq 2$,
\begin{equation}\label{eq:2.6-aa}
\mathbf{P}\left(\bigcap_{j=1}^{k}\widetilde{C}_m^j\right)
=
\mathbf{E}\left[\mathbf{1}_{\bigcap_{j=1}^{k-1}\widetilde{C}_m^j}
\mathbf{P}\left(\widetilde{C}_m^k\mid\mathcal{F}_m^{k-1}\right)\right]
\geq
\gamma_+ \mathbf{P}\left(\bigcap_{j=1}^{k-1}\widetilde{C}_m^j\right)
\geq
\gamma_+^{k-1}\mathbf{P}(\widetilde{C}_m^1)
=
\gamma_+^{k-1}.
\end{equation}
When $k=1$, the display above is immediate.

By the definitions of $\{\widetilde{C}_m^i, 1\leq i\leq k\}$ and $\{T_m^j, 1\leq j\leq k\}$,
one can easily check by induction that for any $1\leq j\leq k$,
\begin{align}\label{eq:2.7-aa}
\bigcap_{i=1}^{j}\widetilde{C}_m^i\subseteq\{T_m^j<T_{m+1}^1\}.
\end{align}
Combining (\ref{eq:2.3-a}),  \eqref{eq:2.6-aa} and \eqref{eq:2.7-aa}, we obtain
\begin{equation}\label{eq:2.8-aa}
\mathbf{P}(C_m^k)\geq\gamma_+^{k-1}>0,~~ \forall m\geq1.
\end{equation}

Using the transience of $\{S_n\}_{n\ge0}$ and Lemma~\ref{lem:2.1}, one can easily see that the favorite sites process $(\mathcal{K}(n))_{n\ge1}$ is transient, i.e.,  no fixed site can be a favorite site infinitely often.

Now we prove \eqref{thm:1.2-a}. Following the argument in \cite[Section~3.3]{DS2018}, we fix $N\geq1$ and  define
$$
\bar{S}_n^N:=S_{N+n}-S_N,
~~
\bar{\xi}^N(x,n):=\sum_{i=0}^n\mathbf{1}_{\{\bar{S}_i^N=x\}},\ n\geq 0.
$$
Let
$$
\bar{\mathcal{K}}^N(n)
:=
\left\{x\in\mathbb{Z}:\bar{\xi}^N(x,n)
=
\max_{y\in\mathbb{Z}}\bar{\xi}^N(y,n)\right\}.
$$
Since $(\bar{S}_n^N)_{n\geq0}$ has the same law as $(S_n)_{n\geq0}$, its favorite-site process is also transient. Therefore, almost surely, for all sufficiently large $n$,
$$
\mathcal{K}(N+n)\cap\{S_0,\ldots,S_{N-1}\}=\varnothing
~~\text{and}~~
\bar{\mathcal{K}}^N(n)\cap\{S_0-S_N,\ldots,S_{N-1}-S_N\}=\varnothing.
$$
For every $x\notin\{S_0-S_N,\ldots,S_{N-1}-S_N\},$ we have $\xi(S_N+x,N+n)=\bar{\xi}^N(x,n).$
Consequently, almost surely, for all sufficiently large $n$,
$\mathcal{K}(N+n)=\left\{S_N+x:x\in\bar{\mathcal{K}}^N(n)\right\}.$
It follows that
$$
\mathbf{P}\left(\left\{\#\mathcal{K}(n)=k\text{ i.o.}\right\}
\mathbin{\triangle}
\left\{\#\bar{\mathcal{K}}^N(r)=k\text{ i.o.}\right\}\right)=0.
$$
Moreover,
$
\left\{\#\bar{\mathcal{K}}^N(n)=k\text{ i.o.}\right\}\in\sigma(X_{N+1},X_{N+2},\ldots).
$
Hence
$$
\limsup_{N\to\infty}\left\{\#\bar{\mathcal{K}}^N(n)=k\text{ i.o.}\right\}
\in\bigcap_{N\geq1}\sigma(X_{N+1},X_{N+2},\ldots),
$$
and
$$
\mathbf{P}\left(\left\{\#\mathcal{K}(n)=k\text{ i.o.}\right\}
\mathbin{\triangle}
\limsup_{N\to\infty}\left\{\#\bar{\mathcal{K}}^N(n)=k\text{ i.o.}\right\}\right)=0.
$$
Thus $\{\#\mathcal{K}(n)=k\text{ i.o.}\}$ belongs to the completed tail $\sigma$-field of $(X_n)_{n\geq1}$. Kolmogorov's 0-1 law and \eqref{eq:2.8-aa} yield
$\mathbf{P}\left(\#\mathcal{K}(n)=k\text{ i.o.}\right)=1.$
The proof is complete.
\end{proof}

\subsection{Proof of Theorem \ref{thm:1.3}}\label{sec:3}

Throughout  this subsection, we assume that
$\sigma^2:=\mathbf{E}[X_1^2]<\infty$.
To prove Theorem \ref{thm:1.3}, we need four lemmas (Lemmas 2.3, 2.4, 2.6 and 2.7).
Lemmas \ref{pro:2.3} and \ref{lem:2.4} play important roles in the proof of Theorem \ref{thm:1.3}.

\begin{lem}\label{pro:2.3}
There exist constants $C>0$ and $a\in(1, 1+\sqrt{1-\gamma})$ such that for any  integer
$u\geq 0$,
\begin{equation}\label{eq:2.9}
\sup_{z\in\mathbb{Z}\setminus\{0\}}\mathbf{P}\left(\xi(0,\infty)+\xi(z,\infty)\geq u\right)
\leq
C\left(1-\frac{\gamma}{a}\right)^u.
\end{equation}
\end{lem}

\begin{proof}[\bf Proof]
 Obviously, we need only to prove (\ref{eq:2.9}) for $u>0$.
This proof is based on the finite set occupation time method of Cs\'aki et al.~\cite{CFRRS2005}. We adapt their Green matrix spectral argument to our setting and establish a uniform estimate.

Let
$G(x):=\sum_{n=0}^{\infty}\mathbf{P}(S_n=x),~ x\in\mathbb{Z}$,
be the Green function of $\{S_n\}_{n\geq 0}$.
By \cite[Proposition 1.4]{S2001},
\begin{align}\label{eq:2.9-1}
G(0)=\mathbf{E}[\xi(0,\infty)]=\frac{1}{1-\mathbf{P}(\hat{\tau}_0<\infty)}=\frac{1}{\gamma}.
\end{align}
For $x\neq0$, the first entrance decomposition \cite[Proposition 1.5; Eq.(4), p.34]{S2001} and spatial homogeneity give
\begin{equation}\label{eq:2.10}
G(x)
=
\mathbf{P}(\hat{\tau}_x<\infty)G(0)
=
\frac{\mathbf{P}(\hat{\tau}_x<\infty)}{\gamma}.
\end{equation}

Fix $z\in\mathbb{Z}\setminus\{0\}$. Define the successive hitting times of the set $\{0,z\}$ by
$$
\bar{\tau}_z^0:=0,
~~~
\bar{\tau}_z^{k+1}:=\inf\left\{n>\bar{\tau}_z^k:S_n\in\{0,z\}\right\},~~k\geq 0.
$$
We use the convention that $\inf\emptyset=\infty$, and if $\bar{\tau}_z^k=\infty$, we  define  $\bar{\tau}_z^{k+1}:=\infty$. Let $Q_z$ be the sub-stochastic matrix,  indexed by $\{0,z\}$, defined by
$$
Q_z(x,y)
:=
\mathbf{P}_x\left(S_{\bar{\tau}_z^1}=y,\,\bar{\tau}_z^1<\infty\right),
~~x,y\in\{0,z\}.
$$
The strong Markov property at $\bar{\tau}_z^k$ gives
$$
\begin{aligned}
&~~~~\mathbf{P}_x\left(S_{\bar{\tau}_z^{k+1}}=y,\,\bar{\tau}_z^{k+1}<\infty\right)\\
&=
\mathbf{E}_x\left[\mathbf{1}_{\{\bar{\tau}_z^k<\infty\}} \mathbf{P}_{S_{\bar{\tau}_z^k}} \left(S_{\bar{\tau}_z^1}=y,\,\bar{\tau}_z^1<\infty\right)\right]
=
\mathbf{E}_x\left[ \mathbf{1}_{\{\bar{\tau}_z^k<\infty\}}Q_z(S_{\bar{\tau}_z^k},y)\right]\\
&=
\mathbf{E}_x\left[\mathbf{1}_{\{\bar{\tau}_z^k<\infty\}}
\sum_{w\in\{0,z\}}\mathbf{1}_{\{S_{\bar{\tau}_z^k}=w\}}Q_z(w,y)\right]
=
\sum_{w\in\{0,z\}}\mathbf{E}_x\left[
\mathbf{1}_{\{S_{\bar{\tau}_z^k}=w,~\bar{\tau}_z^k<\infty\}}\right]Q_z(w,y)\\
&=
\sum_{w\in\{0,z\}}
\mathbf{P}_x\left(S_{\bar{\tau}_z^k}=w,~\bar{\tau}_z^k<\infty\right)Q_z(w,y).
\end{aligned}
$$
Induction on $k$  gives
\begin{equation}\label{eq:2.11}
Q_z^{(k)}(x,y)
=
\mathbf{P}_x\left(S_{\bar{\tau}_z^k}=y,\,\bar{\tau}_z^k<\infty\right),~~k\geq1.
\end{equation}

For every $x,y\in\{0,z\}$,
$\xi(y,\infty)
=
\mathbf{1}_{\{x=y\}}
+
\sum_{k=1}^{\infty}\mathbf{1}_{\{\bar{\tau}_z^k<\infty,\, S_{\bar{\tau}_z^k}=y\}},
~
\mathbf{P}_x\text{-a.s..}
$
Thus by  \eqref{eq:2.11} and spatial homogeneity, we have
$$
\begin{aligned}
\mathbf{1}_{\{x=y\}}+\sum_{k=1}^{\infty}Q_z^{(k)}(x,y)
=
\mathbf{E}_x[\xi(y,\infty)]
=
\sum_{n=0}^{\infty}\mathbf{P}_x(S_n=y)
=
G(y-x).
\end{aligned}
$$
Consequently,
$$
I+\sum_{k=1}^{N}Q_z^{(k)} \to
\begin{pmatrix}
G(0)&G(z)\\
G(-z)&G(0)
\end{pmatrix},
~~N\to\infty,
$$
and hence
$$
Q_z^{(N+1)}
=
\left(I+\sum_{k=1}^{N+1}Q_z^{(k)}\right)-\left(I+\sum_{k=1}^{N}Q_z^{(k)}\right)
\to 0.
$$
Moreover,
$$
(I-Q_z)\left(I+\sum_{k=1}^{N}Q_z^{(k)}\right)
=\left(I+\sum_{k=1}^{N}Q_z^{(k)}\right)(I-Q_z)
=I-Q_z^{(N+1)}.
$$
Letting $N\to\infty$, we obtain
\begin{equation}\label{eq:2.12}
(I-Q_z)^{-1}=I+\sum_{k=1}^{\infty}Q_z^{(k)}
=
\begin{pmatrix}
G(0)&G(z)\\
G(-z)&G(0)
\end{pmatrix}.
\end{equation}

Let $\rho_z$ and  $\rho(Q_z^{(N)})$  denote the spectral radius of $Q_z$ and $Q_z^{(N)}$ respectively. The convergence $Q_z^{(N)}\to0$ gives
$$
\rho_z^N=\rho(Q_z^{(N)})\leq \|Q_z^{(N)}\|_{\infty}\to0,
$$
and hence $\rho_z<1$. Since $Q_z$ is a finite non-negative matrix, the Perron-Frobenius theorem \cite[Theorem~1.28]{BM2005} shows that  $\rho_z$ is a non-negative eigenvalue of $Q_z$.  If $\eta$ is any eigenvalue of $Q_z$, then $(1-\eta)^{-1}$ is an
eigenvalue of $(I-Q_z)^{-1}$ and
$$
\left|\frac{1}{1-\eta}\right|
\leq
\frac{1}{1-|\eta|}
\leq
\frac{1}{1-\rho_z}.
$$
Equality holds for $\eta=\rho_z$. It follows that $\frac{1}{1-\rho_z}$ is the largest eigenvalue of $(I-Q_z)^{-1}$. By \eqref{eq:2.12}, the characteristic polynomial of $(I-Q_z)^{-1}$ is
$$
\det
\begin{pmatrix}
G(0)-t&G(z)\\
G(-z)&G(0)-t
\end{pmatrix}
=(G(0)-t)^2-G(z)G(-z).
$$
Its eigenvalues are $G(0)\pm\sqrt{G(z)G(-z)}.$ Therefore,
$\frac{1}{1-\rho_z}=G(0)+\sqrt{G(z)G(-z)}.$
Thus by (\ref{eq:2.9-1}) and \eqref{eq:2.10}, we have
\begin{equation}\label{eq:2.13}
\rho_z
=
1-\frac{\gamma}{1+\sqrt{\mathbf{P}(\hat{\tau}_z<\infty)\mathbf{P}(\hat{\tau}_{-z}<\infty)}}.
\end{equation}

We next prove
\begin{equation}\label{eq:2.14}
\kappa
:=
\sup_{z\geq1}\sqrt{\mathbf{P}(\hat{\tau}_z<\infty)\mathbf{P}(\hat{\tau}_{-z}<\infty)}
<\sqrt{1-\gamma}.
\end{equation}
Since $\mathbf{P}(\hat{\tau}_0<\infty)=1-\gamma>0$, there exist $r\geq1$ and $y_1,\ldots,y_r\in\mathbb{Z}$ such that
$$
c_0:=\mathbf{P}(X_1=y_1,\ldots,X_r=y_r)>0,\ \mbox{and}\ \sum_{i=1}^r y_i=0.
$$
Let $s_j:=\sum_{i=1}^j y_i,~0\leq j\leq r,$ with $s_0=s_r=0$, and choose $q\in\{0,\ldots,r-1\}$ such that $s_q=\max_{0\leq j\leq r-1}s_j.$
Then $(y_{q+1},\ldots,y_r,y_1,\ldots,y_q)$ has all partial sums nonpositive, since they are equal to $s_j-s_q $ or $s_j-s_q+s_r.$
Thus by permuting the indices, we may assume that $\sum_{i=1}^j y_i\leq0, \forall 1\leq j\leq r.$
Hence, for every $z\geq1$,
$$
\{X_1=y_1,\ldots,X_r=y_r\}\cap\{S_{r+n}\neq0,\ \forall n\geq1\}
\subseteq
\{\hat{\tau}_0<\infty\}\setminus\{\hat{\tau}_z<\infty,\ S_n=0\ \text{for some}\ n>\hat{\tau}_z\}.
$$
Therefore, by the strong Markov property at $\hat{\tau}_z$, spatial homogeneity and independence, we have
$$
\begin{aligned}
\mathbf{P}(\hat{\tau}_z<\infty)\mathbf{P}(\hat{\tau}_{-z}<\infty)
&=\mathbf{E}[\mathbf{1}_{\{\hat{\tau}_z<\infty\}} \mathbf{P}(S_n=0 ~\text{for some}~n>\hat{\tau}_z \ |  \ \mathcal{F}_{\hat{\tau}_z})]\\
&=\mathbf{P}\left(\hat{\tau}_z<\infty,\,S_n=0~\text{ for some}~ n>\hat{\tau}_z\right)\\
&\leq
\mathbf{P}(\hat{\tau}_0<\infty)
-
\mathbf{P}\left(X_1=y_1,\ldots,X_r=y_r,\,S_{r+n}\neq0,\ n\geq1\right)\\
&=
1-\gamma-c_0\gamma,
\end{aligned}
$$
where
$\mathcal{F}_{\hat{\tau}_z}
:=
\left\{A\in\mathcal{F}_\infty:A\cap\{\hat{\tau}_z\leq n\} \in\mathcal{F}_n, \ \forall \, n\geq0 \right\}.$
Taking the supremum over $z\geq1$ gives
$\kappa\leq\sqrt{1-\gamma-c_0\gamma}<\sqrt{1-\gamma},$
which implies  \eqref{eq:2.14}. By symmetry in $z$ and $-z$, the same bound holds uniformly for all $z\in\mathbb{Z}\setminus\{0\}$.

Choose $a\in (1+\kappa, 1+\sqrt{1-\gamma})$.
By \eqref{eq:2.13}-\eqref{eq:2.14} and symmetry in $z$ and $-z$, we have
$\sup_{z\in\mathbb{Z}\setminus\{0\}}\rho_z\leq1-\frac{\gamma}{1+\kappa}<1-\frac{\gamma}{a}<1.$
Let $\eta_1$ and $\eta_2$ be the eigenvalues of $Q_z$. Then
\begin{equation}\label{eq:2.15}
|\eta_1|\vee|\eta_2|\leq\rho_z\leq 1-\frac{\gamma}{1+\kappa}, ~~ \forall z\in\mathbb{Z}\setminus\{0\}.
\end{equation}
Suppose first that $\eta_1\neq\eta_2$. The Cayley-Hamilton theorem \cite[Theorem~A.10]{BM2005} gives
$$
Q_z^{(m)}
=
\frac{\eta_1^m-\eta_2^m}{\eta_1-\eta_2}Q_z
-\eta_1\eta_2 \frac{\eta_1^{m-1}-\eta_2^{m-1}}{\eta_1-\eta_2}I,~~m\geq1.
$$
Since $Q_z$ is sub-stochastic, $\|Q_z\|_{\infty}\leq1.$
It follows from \eqref{eq:2.15} and $\frac{\eta_1^m-\eta_2^m}{\eta_1-\eta_2}=\sum_{j=0}^{m-1}\eta_1^{m-1-j}\eta_2^j$ that
\begin{align}\label{eq:2.15-1}
\|Q_z^{(m)}\|_{\infty}
\leq
m\left(1-\frac{\gamma}{1+\kappa}\right)^{m-1}
+(m-1)\left(1-\frac{\gamma}{1+\kappa}\right)^m
\leq
C(m+1)\left(1-\frac{\gamma}{1+\kappa}\right)^m.
\end{align}
If $\eta_1=\eta_2=:\eta$, then  the Cayley--Hamilton theorem \cite[Theorem~A.10]{BM2005} yields, for every $m\geq2$,
$$
Q_z^{(m)}=\eta^mI+m\eta^{m-1}(Q_z-\eta I),
$$
and the same estimate \eqref{eq:2.15-1} follows from \eqref{eq:2.15} and $\|Q_z\|_{\infty}\leq1$.

The inequality $0<\frac{1-\gamma/(1+\kappa)}{1-\gamma/a}<1$ implies
$
\sup_{m\geq0}m\left(\frac{1-\gamma/(1+\kappa)}{1-\gamma/a}\right)^{m-1} <\infty,
$
which together with (\ref{eq:2.15-1}) implies that there exists $C>0$ such that for all $z\in\mathbb{Z}\setminus\{0\}$ and $m\geq1$,
\begin{align}\label{eq:2.16}
\|Q_z^{(m-1)}\|_{\infty}
\leq C \left[ \sup_{m\geq1} m\left(\frac{1-\gamma/(1+\kappa)}{1-\gamma/a}\right)^{m-1}\right]  \left(1-\frac{\gamma}{a}\right)^{m-1}
\leq
C\left(1-\frac{\gamma}{a}\right)^{m-1}.
\end{align}
For every integer $m\geq1$, \eqref{eq:2.11} gives
\begin{align}\label{eq:2.16-1}
\mathbf{P}\left(\xi(0,\infty)+\xi(z,\infty)\geq m\right)
=
\mathbf{P}\left(\bar{\tau}_z^{m-1}<\infty\right)
=
\sum_{y\in\{0,z\}}Q_z^{(m-1)}(0,y)\leq  \|Q_z^{(m-1)}\|_{\infty}.
\end{align}
For $u>0$, take $m=\lceil u\rceil$. By \eqref{eq:2.16} and \eqref{eq:2.16-1}, we get
$$
\sup_{z\in\mathbb{Z}\setminus\{0\}}\mathbf{P}\left(\xi(0,\infty)+\xi(z,\infty)\geq u\right)
\leq
C\left(1-\frac{\gamma}{a}\right)^{\lceil u\rceil-1}
\leq
C\left(1-\frac{\gamma}{a}\right)^u.
$$
Hence \eqref{eq:2.9} holds and the proof is complete.
\end{proof}

Fix $\varepsilon\in(0,1)$. Recall that $\lambda=-1/\log(1-\gamma)$. For $n\geq 1$, define
$$
D_n:=
\left\{
\left(1-\frac{\varepsilon}{2}\right)\lambda\log n
\leq
\xi^*(n)
\leq
\left(2+\frac{\varepsilon}{2}\right)\lambda\log n
\right\}.
$$
To exclude  pairs of distinct thick points occurring within a short time interval, we set
\begin{align}\label{eq:2.16-2}
\bar{F}_n
:=
\sum_{i=1}^n\sum_{j=i+1}^{(i+(\lambda\log n)^2)\wedge n}
\mathbf{1}_{\left\{\min\{\xi(S_i,n),\xi(S_j,n)\}
\geq
(1-\varepsilon)\lambda\log n,\,S_i,S_j\notin S_{(i,j)},\,S_i\ne S_j\right\}},
\end{align}
and let
$$
E_n:=\{\bar{F}_n=0\}.
$$
Thus, on $E_n$, no two distinct thick points occur at times separated by no more than $(\lambda\log n)^2$.

The proofs of Lemmas~\ref{lem:2.4} and~\ref{lem:2.6} below follow the  method of \cite{HLOZ2025} for simple symmetric random walk in dimensions $d\geq3$ with suitable modifications.
The key estimate controls the probability that two distinct thick points occur within a time interval of length of order $(\log n)^2$. The details are postponed to Section 2.3.

Define
\begin{equation}\label{lem-2.4-0}
\delta:=-2\lambda\log\left(1-\frac{\gamma}{a}\right)-1,
\end{equation}
where $\lambda=-1/\log(1-\gamma)$ and $a$ is the positive constant in Lemma 2.3.  By the choice of $a$ in Lemma \ref{pro:2.3}, we know that  $\delta>0$.

\begin{lem}\label{lem:2.4}
For any $\iota>0$ satisfying $\iota\delta/2>1+\delta$ and any  $\varepsilon\in(0,1)$ satisfying $(1+\delta)(1-\varepsilon)^2>1+\frac{\delta}{2}$,
there exists a constant $C=C(\varepsilon,\iota)<\infty$ such that, for all sufficiently large $N$,
$$
\mathbf{P}\left(\bigcup_{n\ge N}(D_n\cap E_n)^c\right)\le C N^{-[(\varepsilon/2)\wedge(\delta/\iota)]}.
$$
Consequently,
$$
\mathbf{P}\left(\liminf_{n\to\infty} (D_n\cap E_n)\right)=1.
$$
\end{lem}

\begin{rmk}\label{rem:2.5}
On $D_n\cap E_n$, if $\xi^*(n)=m$, then
$$
\left(1-\frac{\varepsilon}{2}\right)\lambda\log n
\leq
m
\leq
\left(2+\frac{\varepsilon}{2}\right)\lambda\log n .
$$
Therefore, for all sufficiently large $n$,
$
(1-\varepsilon)\lambda\log n\leq m, ~\frac{m^2}{16}<(\lambda\log n)^2 .
$
We claim that, on $D_n\cap E_n\cap\{\xi^*(n)=m\}$, for every $k\geq2$ satisfying $T_m^k\leq n$,
\begin{equation*}
T_m^k-T_m^{k-1}\geq\frac{m^2}{16}.
\end{equation*}
Indeed, if $T_m^k-T_m^{k-1}< m^2/16$, then the path must pass from $L_m^{k-1}$ to $L_m^k$ during the interval $[T_m^{k-1},T_m^k]$. Hence there exist $i<j$ in this interval such that
$$
j-i<\frac{m^2}{16}<(\lambda\log n)^2,~~S_i\ne S_j,~~S_i,S_j\notin S_{(i,j)}.
$$
Moreover,
$\xi(S_i,n)\geq m,~~\xi(S_j,n)\geq m .$
Thus this pair is counted by $\bar{F}_n$, contradicting $E_n=\{\bar{F}_n=0\}$.
\end{rmk}

\medskip

\begin{lem}\label{lem:2.6}
Almost surely, there exists $M=M(\omega)<\infty$ such that for every
$m>M$,
$$
\mathbf{1}_{C_m^k}(\omega)=\mathbf{1}_{B_m^k}(\omega),~~k\geq2,
~~\text{and}~~
T_m^1>\exp\left(\frac{m}{4\lambda}\right).
$$
\end{lem}

\begin{lem}\label{lem:2.7}
There exists a constant $C=C(\mu,\sigma)<\infty$ such that for all sufficiently large $m$,
\begin{align}\label{lem:2.7-a}
 \mathbf{P}\left(S_n\leq0\ \mbox{for some}\  n\geq \left\lfloor m^2/16\right\rfloor \right)\leq Cm^{-2}.
\end{align}
\end{lem}

\begin{proof}[\bf Proof]
Set $M_1:=\lfloor m^2/16\rfloor.$ The process $(S_n-n\mu)_{n\geq0}$ is a square integrable martingale. Moreover, for every $N\geq1$,
$$
\begin{aligned}
\mathbf{E}\left[(S_N-N\mu)^2\right]
&=
\mathbf{E}\left[\left(\sum_{i=1}^N (X_i-\mu)\right)^2\right]
=
\sum_{i=1}^N\mathbf{E}\left[(X_i-\mu)^2\right]\\
&=
N\mathbf{E}\left[(X_1-\mu)^2\right]\leq N\sigma^2 .
\end{aligned}
$$
Kolmogorov's maximal inequality therefore yields, for any $N\geq1$ and $b>0$,
$$
\mathbf{P}\left(\max_{1\leq t\leq N}|S_t-t\mu|\geq b\right)
\leq \frac{\mathbf{E}\left[(S_N-N\mu)^2\right]}{b^2}
\leq \frac{N\sigma^2}{b^2}.
$$
For $r\geq0$, if $S_n\leq0$ for some $n\in[2^rM_1,2^{r+1}M_1)$, then $|S_n-n\mu|\geq n\mu\geq 2^rM_1\mu .$
Therefore,
$$
\begin{aligned}
\mathbf{P}\left(S_n\leq0\ \mbox{for some}\ n\geq M_1\right)
&\leq
\sum_{r=0}^{\infty}\mathbf{P}\left(S_n\leq0\ \mbox{for some}\ n\in[2^rM_1,2^{r+1}M_1)\right)\\
&\leq
\sum_{r=0}^{\infty} \mathbf{P} \left(\max_{1\leq t\leq 2^{r+1}M_1}|S_t-t\mu|\geq2^rM_1\mu \right)\\
&\leq
\sum_{r=0}^{\infty}\frac{2^{r+1}M_1\sigma^2}{(2^rM_1\mu)^2}
=
\frac{\sigma^2}{\mu^2M_1}\sum_{r=0}^{\infty}2^{1-r}
\leq
\frac{C}{M_1}.
\end{aligned}
$$
Substituting $M_1$  by $\lfloor m^2/16\rfloor$, we get (\ref{lem:2.7-a}). The proof is complete.
\end{proof}

\begin{proof}[\noindent{\bf Proof of Theorem~\ref{thm:1.3}}]
For $m\geq1$, set
$$
\mathcal{G}_m:=\sup\{k:T_m^k<T_{m+1}^1\}.
$$
By Lemma~\ref{lem:2.1}, it suffices to prove
\begin{equation}\label{eq:2.17}
 \limsup_{m\to\infty}\frac{\mathcal{G}_m}{\log m}
=-\frac{1}{\log\gamma}=:\theta ~~\text{a.s.}.
\end{equation}
Fix $\varepsilon \in(0,1)$ satisfying the assumption of Lemma~\ref{lem:2.4}. For
$m\ge1$, set
$$
K_m:=\bigcap_{n\ge \exp(m/(4\lambda))}(D_n\cap E_n).
$$

We first prove the upper bound in \eqref{eq:2.17}. Put
$I_m:=\lfloor(1+\varepsilon)\theta\log m\rfloor.$
On $K_m\cap\{\mathcal{G}_m>I_m\}$, the event $C_m^{I_m+1}$ occurs. Lemma~\ref{lem:2.6} implies that, almost surely, for all sufficiently large $m$, and every $1\leq j\leq I_m$,
$$
\exp\left(\frac{m}{4\lambda}\right) <T_m^1 \leq T_m^{j+1}<T_{m+1}^1,
~~
\xi^*(T_m^{j+1})=m.
$$
Hence, by the definition of $K_m$ and Remark~\ref{rem:2.5}, for every $1\leq j\leq I_m$,
$
T_m^{j+1}-T_m^j\geq\frac{m^2}{16}.
$
Moreover, since ${L_m^j}$ has local time $m$ at time $T_m^j$, it
cannot be visited again before $T_{m+1}^1$. Using
$T_m^j+\left\lfloor m^2/16 \right\rfloor \leq T_m^{j+1}<T_{m+1}^1,$
we have
$
S_{T_m^j}\notin
\left\{S_{T_m^j+1},\ldots,S_{T_m^j+\left\lfloor m^2/16\right\rfloor}\right\}.
$
For $1\leq j\leq I_m$, define
$$
R_m^j
:=
\left\{T_m^{j+1}<\infty,~T_m^{j+1}-T_m^j\geq\left\lfloor m^2/16 \right\rfloor,~
S_{T_m^j}\notin\left\{S_{T_m^j+1},\ldots,S_{T_m^j+\left\lfloor m^2/16\right\rfloor}\right\}
\right\}.
$$
Consequently,
\begin{equation}\label{eq:limsup}
\limsup_{m\to\infty}\left(K_m\cap\{\mathcal{G}_m>I_m\}\right)
\subseteq
\limsup_{m\to\infty}\cap_{j=1}^{I_m}R_m^j
~~
\text{a.s.}.
\end{equation}
Since $\bigcap_{i=1}^{j-1}R_m^i\in\mathcal{F}_m^j$, applying the strong Markov property at $T_m^j$, successive conditioning, and Lemma \ref{lem:2.7}, we get
that  for all sufficiently large $m$,
$$
\begin{aligned}
\mathbf{P}\left(\cap_{j=1}^{I_m}R_m^j\right)
&=
\mathbf{E}\left[\mathbf{1}_{\bigcap_{j=1}^{I_m-1}R_m^j}
\mathbf{P}\left(R_m^{I_m}\mid\mathcal{F}_m^{I_m}\right)\right]\\
&\leq
\mathbf{P}\left(S_0\notin
\left\{S_1,\ldots,S_{\left\lfloor m^2/16\right\rfloor}\right\}\right)
\mathbf{P}\left(\cap_{j=1}^{I_m-1}R_m^j\right)
\leq
\mathbf{P}\left(S_0\notin
\left\{S_1,\ldots,S_{\left\lfloor m^2/16\right\rfloor}\right\}\right)^{I_m}\\
&\leq
\left[\mathbf{P}\left(S_n\neq0,\ n\geq1\right)
+
\mathbf{P}\left(S_n=0\text{ for some }n\geq \lfloor m^2/16\rfloor\right)\right]^{I_m}\\
&\leq
\left(\gamma+Cm^{-2}\right)^{I_m}
\leq
Cm^{-(1+\varepsilon)}.
\end{aligned}
$$
Thus, by the Borel--Cantelli lemma,
$\mathbf{P}\left(\bigcap_{j=1}^{I_m}R_m^j~\text{i.o.}\right)=0.$

By Lemma~\ref{lem:2.4}, the definition of $K_m$, and the Borel-Cantelli lemma,
$\mathbf{P}\left(K_m^{\mathrm c}\text{ i.o.}\right)=0.$
 By \eqref{eq:limsup}, we have
$$
\left\{\mathcal{G}_m>I_m\text{~ i.o.}\right\}
\subseteq
\left\{K_m^{\mathrm c}\text{~ i.o.}\right\}
\cup
\left\{\cap_{j=1}^{I_m}R_m^j\text{~i.o.}\right\}~~\text{a.s.}.
$$
Consequently,
$
\mathbf{P}\left(\mathcal{G}_m>I_m~\text{i.o.}\right)=0.
$
 Therefore, almost surely,
\begin{equation}\label{eq:2.21}
\limsup_{m\to\infty}\frac{\mathcal{G}_m}{\log m}
\leq
\lim_{m\to\infty}\frac{I_m}{\log m}
=
(1+\varepsilon)\theta\to\theta,\ \text{as}\ \varepsilon\to 0.
\end{equation}

We next prove the lower bound in \eqref{eq:2.17}.
In the remainder of  this proof,
we always work with $m$  sufficiently large
so that $\gamma>cm^{-2}$.
Put $J_m:=\lfloor(1-\varepsilon)\theta\log m\rfloor$. Fix $3\leq k\leq J_m$. By  the definitions of $C_m^{k-1}$ and $A_m^k$, we have
$$
C_m^{k-1}\cap A_m^k\supseteq
C_m^{k-1}\cap\left\{S_{T_m^{k-1}+n}\neq L_m^{k-1},\ \forall n\geq1\right\}
\cap
\left\{S_{T_m^{k-1}+n}\notin\{L_m^1,\ldots,L_m^{k-2}\},\ \forall n\geq\left\lfloor m^2/16 \right\rfloor\right\}.
$$
Then by $C_m^{k-1}\cap\bigcap_{j=1}^{k-2}\{L_m^j<L_m^{k-1}\}\in \mathcal{F}_m^{k-1}$, the strong Markov property at $T_m^{k-1}$, spatial homogeneity, and Lemma~\ref{lem:2.7}, we get
\begin{align*}
&1_{C_m^{k-1}\cap\bigcap_{j=1}^{k-2}\{L_m^j<L_m^{k-1}\}}\mathbf{P}\left(\left. A_m^k\right|\mathcal{F}_m^{k-1}\right)
=\mathbf{P}\left(\left. C_m^{k-1}\cap\cap_{j=1}^{k-2}\{L_m^j<L_m^{k-1}\}\cap A_m^k\right| \mathcal{F}_m^{k-1}\right)\\
&\geq \mathbf{P}\left(C_m^{k-1}\cap\cap_{j=1}^{k-2}\{L_m^j<L_m^{k-1}\}\right.\\
&\quad\quad\quad \left.\left.\cap \left\{S_{T_m^{k-1}+n}\neq L_m^{k-1},\ \forall n\geq1\right\}
\cap \left\{S_{T_m^{k-1}+n}\notin\{L_m^1,\ldots,L_m^{k-2}\},\ \forall n\geq\left\lfloor m^2/16 \right\rfloor\right\}\right| \mathcal{F}_m^{k-1}\right)\\
&\geq \mathbf{P}\left(C_m^{k-1}\cap\cap_{j=1}^{k-2}\{L_m^j<L_m^{k-1}\}\right.\\
&\quad\quad\quad \left.\cap
\left\{S_{T_m^{k-1}+n}\neq L_m^{k-1},\, \forall n\geq1\right\}
\cap
\left\{ S_{T_m^{k-1}+n}-L_m^{k-1}\geq0,\ \forall n\geq \left\lfloor m^2/16 \right\rfloor \right\} \, \middle| \, \mathcal{F}_m^{k-1}\right)\\
&=1_{C_m^{k-1}\cap\cap_{j=1}^{k-2}\{L_m^j<L_m^{k-1}\}}\mathbf{P}
\left(\left\{S_{T_m^{k-1}+n}\neq L_m^{k-1},\, \forall n\geq1\right\}
\cap
\left\{ S_{T_m^{k-1}+n}-L_m^{k-1}\geq0,\ \forall n\geq \left\lfloor m^2/16 \right\rfloor \right\} \, \middle| \, \mathcal{F}_m^{k-1}\right)\\
&=1_{C_m^{k-1}\cap\cap_{j=1}^{k-2}\{L_m^j<L_m^{k-1}\}}
\mathbf{P}\left(\{S_n\neq0,\ \forall n\geq1\}
\cap
\left\{S_n\geq0,\ \forall n\geq\left\lfloor m^2/16 \right\rfloor\right\}\right)\\
&\geq1_{C_m^{k-1}\cap\cap_{j=1}^{k-2}\{L_m^j<L_m^{k-1}\}}\left[
\mathbf{P}\left(S_n\neq0,\ \forall n\geq1\right)
-
\mathbf{P}\left(S_n<0\text{ for some }n\geq\left\lfloor m^2/16 \right\rfloor\right)\right]\\
&\geq\left(\gamma-Cm^{-2}\right)1_{C_m^{k-1}\cap\cap_{j=1}^{k-2}\{L_m^j<L_m^{k-1}\}}\quad \text{a.s.},
\end{align*}
which together with the fact that $0<\gamma-Cm^{-2}<\gamma<1$ and $L_m^1,\ldots,L_m^{k-1}$ are distinct on $C_m^{k-1}$, implies that
\begin{align}\label{eq:2.22}
\mathbf{1}_{B_m^{k-1}}\mathbf{P}\left(A_m^k\mid\mathcal{F}_m^{k-1}\right)
&\geq \mathbf{1}_{B_m^{k-1}}\cdot 1_{C_m^{k-1}\cap\cap_{j=1}^{k-2}\{L_m^j<L_m^{k-1}\}}\mathbf{P}\left(A_m^k\mid\mathcal{F}_m^{k-1}\right)
\nonumber\\
&\geq
(\gamma-Cm^{-2})
\mathbf{1}_{B_m^{k-1}\cap C_m^{k-1}\cap\cap_{j=1}^{k-2}\{L_m^j<L_m^{k-1}\}}\nonumber\\
&\geq
(\gamma-Cm^{-2})\mathbf{1}_{B_m^{k-1}}
-
\mathbf{1}_{B_m^{k-1}\cap(C_m^{k-1})^{\mathrm c}}-
\mathbf{1}_{C_m^{k-1}\cap\cup_{j=1}^{k-2}\{L_m^{k-1}<L_m^j\}}\quad\text{a.s.}.
\end{align}
For $k=2$, since $T_m^1<\infty$ a.s. and $\mathbf{P}(B_m^1)=1$, the definition of $A_m^2$ and the strong Markov property at $T_m^1$ give
\begin{align}\label{eq:2.22-A}
\mathbf{P}\left(B_m^2\mid\mathcal{F}_m^1\right)
=
\mathbf{P}\left(A_m^2\mid\mathcal{F}_m^1\right)
\geq
\mathbf{P}\left(S_{T_m^1+n}\neq L_m^1,\ \forall n\geq1\,\middle|\,\mathcal{F}_m^1\right)
=\gamma.
\end{align}

By $B_m^k=B_m^{k-1}\cap A_m^k$,  $B_m^{k-1}\in\mathcal{F}_m^{k-1}$, \eqref{eq:2.22}, \eqref{eq:2.22-A},  and  the definitions of $J_m$ and $\theta$, we get
\begin{align}\label{eq:2.22-B}
&\mathbf{P}\left(B_m^{J_m}\mid\mathcal{F}_m^1\right)
=
\mathbf{E}\left[
\mathbf{1}_{B_m^{J_m-1}}\mathbf{P}\left(A_m^{J_m}\mid\mathcal{F}_m^{J_m-1}\right)
\,\middle|\,\mathcal{F}_m^1\right]\nonumber\\
&\geq
\mathbf{E}\left[(\gamma-Cm^{-2})\mathbf{1}_{B_m^{J_m-1}}
-
\mathbf{1}_{B_m^{J_m-1}\cap(C_m^{J_m-1})^{\mathrm c}}
-
\mathbf{1}_{C_m^{J_m-1}\cap\bigcup_{j=1}^{J_m-2}\{L_m^{J_m-1}<L_m^j\}} \,\middle|\,\mathcal{F}_m^1\right]\nonumber\\
&=
(\gamma-Cm^{-2})\mathbf{P}\left(B_m^{J_m-1}\mid\mathcal{F}_m^1\right)-
\mathbf{E}\left[\mathbf{1}_{B_m^{J_m-1}\cap(C_m^{J_m-1})^{\mathrm c}}
+
\mathbf{1}_{C_m^{J_m-1}\cap\bigcup_{j=1}^{J_m-2}\{L_m^{J_m-1}<L_m^j\}}
\,\middle|\,\mathcal{F}_m^1\right]\nonumber\\
&\geq
\gamma(\gamma-Cm^{-2})^{J_m-2}-
\sum_{k=3}^{J_m}(\gamma-Cm^{-2})^{J_m-k}
\mathbf{E}\left[\mathbf{1}_{B_m^{k-1}\cap(C_m^{k-1})^{\mathrm c}}
+
\mathbf{1}_{C_m^{k-1}\cap\bigcup_{j=1}^{k-2}\{L_m^{k-1}<L_m^j\}}
\,\middle|\,\mathcal{F}_m^1\right]\nonumber\\
&\geq
C^{-1}m^{-(1-\varepsilon)}-
J_m\mathbf{P}\left(\cup_{k=3}^{J_m}
\left[\left(B_m^{k-1}\cap(C_m^{k-1})^{\mathrm c}\right)
\cup
\left(C_m^{k-1}\cap\cup_{j=1}^{k-2}\{L_m^{k-1}<L_m^j\}\right)\right]
\,\middle|\,\mathcal{F}_m^1\right)\quad \text{a.s.}.
\end{align}
By Remark~\ref{rem:2.5} and the definition of $L_m^j$,
for any $1\leq j\leq k-2$, we have 
$$
K_m\cap
\left\{T_m^1>\exp\left(\frac{m}{4\lambda}\right)\right\}\cap C_m^{k-1}\cap\{L_m^{k-1}<L_m^j\}
\subseteq
\left\{T_m^j<\infty,\ S_{T_m^j+n}-S_{T_m^j}<0
\text{ for some }
n\geq\left\lfloor m^2/16\right\rfloor\right\}.
$$
Taking the union over $k$ and $j$, we get
\begin{equation}\label{eq:2.22-a}
\mathbf{1}_{K_m
\cap
\left\{T_m^1>\exp\left(\frac{m}{4\lambda}\right)\right\}
\cap\bigcup_{k=3}^{J_m}
\left[C_m^{k-1}\cap\bigcup_{j=1}^{k-2}\{L_m^{k-1}<L_m^j\}\right]}
\leq
\sum_{j=1}^{J_m-2}\mathbf{1}_{\left\{T_m^j<\infty,\ S_{T_m^j+n}-S_{T_m^j}<0
\text{ for some }
n\geq\left\lfloor m^2/16\right\rfloor\right\}}.
\end{equation}
By the strong Markov property at $T_m^j$, spatial homogeneity, Lemma~\ref{lem:2.7}, and $J_m=O(\log m)$, we have
\begin{equation}\label{eq:2.22-b}
\begin{aligned}
&\sum_{m=1}^{\infty}\sum_{j=1}^{J_m-2}\mathbf{P}\left(T_m^j<\infty,\
S_{T_m^j+n}-S_{T_m^j}<0\text{ for some }n\geq\left\lfloor m^2/16\right\rfloor\right)\\
=&
\sum_{m=1}^{\infty}\sum_{j=1}^{J_m-2}\mathbf{P}(T_m^j<\infty)\mathbf{P}\left(S_n<0
\text{ for some }n\geq\left\lfloor m^2/16\right\rfloor\right)
\leq
C\sum_{m=1}^{\infty}\frac{J_m}{m^2}
<\infty.
\end{aligned}
\end{equation}
Using Lemmas~\ref{lem:2.4} and~\ref{lem:2.6},
\eqref{eq:2.22-a}--\eqref{eq:2.22-b}, and the first Borel--Cantelli lemma, we obtain, a.s.,
\begin{equation*}
\begin{aligned}
&\sum_{m=1}^{\infty}\mathbf{1}_{\bigcup_{k=3}^{J_m}
\left[\left(B_m^{k-1}\cap(C_m^{k-1})^{\mathrm c}\right)
\cup
\left(C_m^{k-1}\cap\bigcup_{j=1}^{k-2}\{L_m^{k-1}<L_m^j\}\right)\right]}\\
\leq&
\sum_{m=1}^{\infty}\left[\mathbf{1}_{\bigcup_{k=3}^{J_m}\left(B_m^{k-1}\cap(C_m^{k-1})^{\mathrm c}\right)}
+
\mathbf{1}_{\bigcup_{k=3}^{J_m}
\left[C_m^{k-1}\cap\bigcup_{j=1}^{k-2}\{L_m^{k-1}<L_m^j\}\right]}\right]\\
\leq&
\sum_{m=1}^{\infty}
\left[\mathbf{1}_{\bigcup_{k=3}^{J_m}
\left(B_m^{k-1}\cap(C_m^{k-1})^{\mathrm c}\right)}+\mathbf{1}_{K_m^{\mathrm c}}
+\mathbf{1}_{\left\{T_m^1\leq\exp\left(\frac{m}{4\lambda}\right)\right\}}
+\sum_{j=1}^{J_m-2}\mathbf{1}_{\left\{T_m^j<\infty,\ S_{T_m^j+n}-S_{T_m^j}<0
\text{ for some }n\geq\left\lfloor\frac{m^2}{16}\right\rfloor\right\}}
\right]<\infty.
\end{aligned}
\end{equation*}
Since both unions are $\mathcal{F}_{m+1}^1$-measurable, the generalized second Borel-Cantelli lemma (\cite[Theorem~4.3.4]{D2019}) give
\begin{equation}\label{eq:2.23}
\sum_{m=1}^{\infty} \mathbf{P}\left(\cup_{k=3}^{J_m}\left[\left(B_m^{k-1}\cap(C_m^{k-1})^{\mathrm c}\right)
\cup
\left(C_m^{k-1}\cap\cup_{j=1}^{k-2}\{L_m^{k-1}<L_m^j\}\right)
\right]\,\middle|\,\mathcal{F}_m^1\right)<\infty,~~\text{a.s.}.
\end{equation}
By the definition of $J_m$, \eqref{eq:2.22-B} and \eqref{eq:2.23}, we obtain
\begin{equation*}
\begin{aligned}
&\sum_{m\leq n}
\mathbf{P}\left(B_m^{J_m}\mid\mathcal{F}_m^1\right)\\
\geq&
C^{-1}\sum_{m\leq n}m^{-(1-\varepsilon)}-J_n\sum_{m=1}^{\infty}
\mathbf{P}\left(\cup_{k=3}^{J_m}
\left[\left(B_m^{k-1}\cap(C_m^{k-1})^{\mathrm c}\right)
\cup
\left(C_m^{k-1}\cap
\cup_{j=1}^{k-2}\{L_m^{k-1}<L_m^j\}\right)\right]\,\middle|\,\mathcal{F}_m^1\right)\\
\geq&
C^{-1}n^\varepsilon-O(\log n)\to\infty,~~\text{a.s.}
\end{aligned}
\end{equation*}
By $B_m^{J_m}\in\mathcal{F}_{m+1}^1$, the generalized second Borel-Cantelli lemma (\cite[Theorem~4.3.4]{D2019}) and Lemma~\ref{lem:2.6}, we get
$\mathbf{P}\left(B_m^{J_m}\text{ i.o.}\right)
=
\mathbf{P}\left(C_m^{J_m}\text{ i.o.}\right)=1.$
By $C_m^{J_m}\subseteq\{\mathcal{G}_m\geq J_m\}$ and the definition of $J_m$, we get
\begin{equation}\label{eq:2.24}
\limsup_{m\to\infty}\frac{\mathcal{G}_m}{\log m}
\geq
\lim_{\varepsilon\downarrow0}\lim_{m\to\infty}\frac{J_m}{\log m}
=
\lim_{\varepsilon\downarrow0}(1-\varepsilon)\theta
=\theta,~~\text{a.s.}.
\end{equation}

Combining \eqref{eq:2.21} and \eqref{eq:2.24} proves \eqref{eq:2.17},   and the proof is complete.
\end{proof}

\subsection{Proofs of Lemmas 2.1, 2.4 and 2.6}\label{sec:4}

\begin{proof}[\bf Proof of Lemma \ref{lem:2.1}]
If $\gamma=1$, then $\mathbf{P}(\hat{\tau}_0<\infty)=0$. By spatial homogeneity and the strong Markov property,  we get that
$\xi^*(n)=1,~n\ge0,~\text{a.s.}$.
Thus $\xi^*(n)/\log n\to0$, which agrees with the convention $\lambda=0$.
Now we deal with the case $0<\gamma<1$.

For every integer $r\geq1$, the strong Markov property gives
\begin{equation}\label{eq:4.1}
\mathbf{P}\left(\xi(0,\infty)\geq r\right)
=\sum_{i\geq r}\mathbf{P}\left(\xi(0,\infty)= i\right)
=(1-\gamma)^{r-1}.
\end{equation}
By spatial homogeneity, the strong Markov property and \eqref{eq:4.1}, we have, for any $\varepsilon\in(0,1)$,
\begin{equation}\label{eq:4.2}
\begin{aligned}
\mathbf{P}\left( \xi^*(n)\geq(1+\varepsilon)\lambda\log n \right)
&\leq
(n+1)\mathbf{P}\left(\xi(0,\infty)\geq(1+\varepsilon)\lambda\log n \right)\\
&\leq
(n+1)(1-\gamma)^{(1+\varepsilon)\lambda\log n-1}
\leq
Cn^{-\varepsilon}.
\end{aligned}
\end{equation}
Taking $n_k=2^k$, we obtain
$$
\sum_{k=1}^{\infty} \mathbf{P}\left(\xi^*(n_k)\geq(1+\varepsilon)\lambda\log n_k\right)<\infty.
$$
The Borel-Cantelli lemma therefore gives that $\xi^*(n_k)<(1+\varepsilon)\lambda\log n_k$ for all sufficiently large $k$, almost surely. For $n_k\leq n<n_{k+1}$, the monotonicity of $\xi^*(n)$ gives
$
\xi^*(n)\leq\xi^*(n_{k+1})<(1+\varepsilon)\lambda\log n_{k+1}.
$
Consequently,
$\frac{\xi^*(n)}{\log n}\leq(1+\varepsilon)\lambda\frac{\log n_{k+1}}{\log n_k},$
and hence
\begin{equation}\label{eq:4.3}
\limsup_{n\to\infty}\frac{\xi^*(n)}{\log n}
\leq
(1+\varepsilon)\lambda,~~\text{a.s.}.
\end{equation}

For the lower bound, set
\begin{equation}\label{eq:4.3-a}
u_n:=\left\lfloor\left(1-{\varepsilon}\right)\lambda\log n\right\rfloor,
~~
v_n:=\left\lfloor\frac{n}{u_n^2}\right\rfloor.
\end{equation}
For $k\ge1$, let $T_k$ be the $k$th return time to the starting point. By the strong Markov property,
\begin{equation}\label{eq:4.3-b}
\mathbf{P}\left(T_1\le u_n,\, T_k-T_{k-1}\le u_n,\ 2\le k\le u_n\right)
=(1-\gamma_{u_n})^{u_n}=:\beta_n.
\end{equation}
On this event, the starting point is visited more than $u_n$ times by time $u_n^2$.
Inside the time interval $[0,n]$, we take the first $v_n=\lfloor n/u_n^2\rfloor$ disjoint subintervals of length $u_n^2$. On each subinterval, starting from its left endpoint, we repeat the preceding construction. The corresponding events depend on disjoint blocks of increments, and hence are independent. By spatial homogeneity and the strong Markov property, each of them has probability $\beta_n$. If at least one of these events occurs, then $\xi^*(n)>u_n$. Therefore,
\begin{equation}\label{eq:4.4}
\mathbf{P}(\xi^*(n)\le u_n)
\leq
(1-\beta_n)^{v_n}.
\end{equation}

Since $\gamma_{u_n}\downarrow\gamma$ and $0<\gamma<1$, we have
$\log(1-\gamma_{u_n})=\log(1-\gamma)+o(1).$ Therefore, by  \eqref{eq:4.3-a}, \eqref{eq:4.3-b} and the definition of $\lambda$, we get
$$
\log\beta_n
=
\left\lfloor\left(1-{\varepsilon}\right)\lambda\log n\right\rfloor\left(\log(1-\gamma)+o(1)\right)
=
-\left(1-{\varepsilon}\right)(1+o(1))\log n.
$$
Thus $\beta_n=n^{-(1-\varepsilon)+o(1)}.$  Consequently,
$$
\beta_n v_n
=
n^{-(1-\varepsilon)+o(1)}\left\lfloor\frac{n}{\left\lfloor(1-\varepsilon)\lambda\log n
\right\rfloor^2}\right\rfloor.
$$
Hence, for every $\alpha\in(0,\varepsilon/2)$, $\beta_n v_n\ge n^\alpha$ for all sufficiently large $n$. Combining this with \eqref{eq:4.4}, we obtain for all sufficiently large $n$,
\begin{equation}\label{eq:4.5}
\mathbf{P}\left(\xi^*(n)\le\left(1-{\varepsilon}\right)\lambda\log n\right)
\leq  (1-\beta_n)^{v_n}
\leq  e^{-\beta_n v_n}
\leq  e^{-n^\alpha}.
\end{equation}
Since $\sum_{n=1}^{\infty}e^{-n^\alpha}<\infty,$ by the Borel-Cantelli lemma,  we have
$
\liminf_{n\to\infty}\frac{\xi^*(n)}{\log n}
\geq
\left(1-{\varepsilon}\right)\lambda,~\text{a.s.}.
$
Letting $\varepsilon\downarrow0$ proves the lower bound.
\end{proof}

\begin{proof}[\bf Proof of Lemma \ref{lem:2.4}.]
Equations \eqref{eq:4.5} and \eqref{eq:4.2} imply
\begin{equation}\label{eq:4.6}
\begin{aligned}
\mathbf{P}(D_n^c)
&=
\mathbf{P}\left( \xi^*(n)<(1-\varepsilon/2)\lambda\log n \ \text{or}\ \xi^*(n)>(2+\varepsilon/2)\lambda\log n\right)\\
&\leq
e^{-n^\alpha}+Cn^{-(1+\varepsilon/2)}
\leq
Cn^{-(1+\varepsilon/2)}.
\end{aligned}
\end{equation}
Recall that $\delta$ is defined in \eqref{lem-2.4-0}.
It follows from Lemma \ref{pro:2.3} and \eqref{lem-2.4-0} that
for any $A\geq 0$ and $n\geq1$,
\begin{equation}\label{eq:4.8}
\sup_{z\in \mathbb{Z}\setminus \{0\}}
\mathbf{P}\left(\xi(0,\infty)+\xi(z,\infty)\ge 2A\lambda\log n\right)
\le
C\left(1-\frac{\gamma}{a}\right)^{2A\lambda\log n}
=
C n^{-A(1+\delta)} .
\end{equation}

We now estimate $\mathbf{P}(E_n^c)$. 
 For $n, \tilde{n}>0$, define
\begin{equation*}
\bar{F}_{n, \tilde{n}}:=\sum_{i=1}^n \sum_{j=i+1}^{(i+ (\lambda \log n)^2)
\wedge n}
\mathbf{1}_{\left\{ \min \left(\xi\left(S_i, n\right), \xi\left(S_j, n\right) \right) \geq (1-\varepsilon) \lambda \log \tilde{n}, S_i, S_j \notin S_{(i, j)}, S_i \neq S_j\right\}}.
\end{equation*}
Note that for any $k \in\left[n^\iota,(n+1)^\iota\right)$, we have
$\bar{F}_k \leq \bar{F}_{(n+1)^\iota, n^\iota}$,  where $\bar{F}_k$ is defined by \eqref{eq:2.16-2}.
Consequently,
\begin{equation}\label{eq:4.9}
 E_k^c \subset\left\{\bar{F}_{(n+1)^\iota, n^\iota}>0\right\}.
\end{equation}

For any non-negative random variables $X, Y$ and  $q>0$, it is straightforward
to see that
\begin{equation}\label{eq:4.10}
\{X+Y \geq q\} \subset \bigcup_{r=0}^{[1 / \varepsilon]}\{X \geq r \varepsilon q,~ Y
\geq [1-(r+1) \varepsilon] q\}.
\end{equation}

Let
$$\hat{S}_l^i:=S_{i-l}-S_i~ ~\text{and}~ ~\widetilde{S}_l^j:=S_{j+l}-S_j.$$
Then, for any $1 \leq i \leq (n+1)^\iota, i < j \leq $ $\left(i+[\lambda \log
((n+1)^\iota)]^2 \right) \wedge(n+1)^\iota$,
by \eqref{eq:4.8} and \eqref{eq:4.10}, we have
\begin{align*}
& \mathbf{P}\left( \min \left(\xi\left(S_i,(n+1)^\iota\right),
\xi\left(S_j,(n+1)^\iota\right)\right) \geq(1-\varepsilon) \lambda \log (n^\iota), S_i,
S_j \notin S_{(i, j)}, S_i \neq S_j\right) \\
\leq &
\sum_{z \in \mathbb{Z}\setminus \{0\}}
\mathbf{P}\left(S_j-S_i=z, \sum_{l=0}^i \mathbf{1}_{\left\{\hat{S}_l^i \in\{0, z\} \right\}}+\sum_{l=0}^{(n+1)^\iota-j}
\mathbf{1}_{\left\{\widetilde{S}_l^j \in\{0,-z\}\right\}} \geq 2(1-\varepsilon) \lambda \log (n^\iota)\right)\\
\le &
\sum_{z \in \mathbb{Z}\setminus\{0\}}
\mathbf{P}\left(S_j-S_i=z\right) \cdot \sum_{r=0}^{\lfloor1 / \varepsilon\rfloor}
\mathbf{P}\left(\sum_{l=0}^i \mathbf{1}_{\left\{\hat{S}_l^i \in\{0, z\}\right\} }
\geq 2 r \varepsilon(1-\varepsilon) \lambda \log (n^\iota)\right) \\
&~~ \cdot \mathbf{P}\left(\sum_{l=0}^{(n+1)^\iota-j}
\mathbf{1}_{\left\{\widetilde{S}_l^j \in\{0,-z\}\right\} }\geq 2[1-(r+1)
\varepsilon](1-\varepsilon) \lambda \log (n^\iota)\right) \\
\leq & \sup_{r\in\{0,\ldots,\lfloor1/\varepsilon\rfloor\}}(\lfloor1/\varepsilon\rfloor+1)\cdot \sup_{z\in \mathbb{Z}\setminus 0}
\mathbf{P}\left( \xi(0, (n+1)^\iota)+\xi(z, (n+1)^\iota)\geq 2 r
\varepsilon(1-\varepsilon) \lambda \log (n^\iota)\right)\\
&~~ \cdot \sup_{z\in \mathbb{Z}\setminus 0}\mathbf{P}\left( \xi(0, (n+1)^\iota)+\xi(z, (n+1)^\iota)\geq 2[1-(r+1) \varepsilon](1-\varepsilon) \lambda \log (n^\iota)\right)\\
\leq &
 \sup_{r\in\{0,\ldots,\lfloor1/\varepsilon\rfloor\}}(\lfloor1/\varepsilon\rfloor+1)
C^2 n^{-\iota r\varepsilon(1-\varepsilon)(1+\delta)}\cdot
n^{-\iota([1-(r+1)\varepsilon]\vee0)(1-\varepsilon)(1+\delta)}
\leq
C n^{-\iota(1-\varepsilon)^2(1+\delta)},
\end{align*}
where in the second inequality we used \eqref{eq:4.10}, and
 in the fourth inequality,  we used \eqref{eq:4.8}.
Put $\nabla_i := (i+[\lambda \log ((n+1)^\iota)]^2) \wedge (n+1)^\iota$.
Summing over $i$, $j$  and using the assumption that $(1-\varepsilon)^2
(1+\delta)>1+\delta/2$ and $\iota \delta/2>1+\delta$, we get
\begin{align}\label{eq:4.11}
&\mathbf{P}\left(\bar{F}_{(n+1)^\iota, n^\iota}>0\right)\notag\\
= & \mathbf{P}\left( \sum_{i=1}^{(n+1)^\iota} \sum_{j=i+1}^{\nabla_i}
\mathbf{1}_{\left\{ \min \left(\xi\left(S_i, (n+1)^\iota\right), \xi\left(S_j,
(n+1)^\iota \right)\right) \geq (1-\varepsilon) \lambda \log (n^\iota), S_i, S_j \notin S_{(i,j)}, S_i \neq S_j\right\}} >0 \right)\notag\\
\leq & \sum_{i=1}^{(n+1)^\iota} \sum_{j=i+1}^{\nabla_i}  \mathbf{P}\left( \min
\left(\xi\left(S_i,(n+1)^\iota\right), \xi\left(S_j,(n+1)^\iota\right)\right) \geq
(1-\varepsilon) \lambda \log (n^\iota), S_i, S_j \notin S_{(i, j)}, S_i \neq
S_j\right)\notag\\
\leq & (n+1)^\iota \cdot  [\lambda \log(n+1)^\iota]^2  \cdot  C n^{- \iota (1-\varepsilon)^2(1+\delta)}
\leq C (\log n)^2 \cdot n^{-\iota\delta/2}
\leq Cn^{-(1+\delta)}.
\end{align}
Combining \eqref{eq:4.9} and \eqref{eq:4.11}, we obtain the bound
\begin{equation}\label{eq:4.12}
\mathbf{P}\left(\bigcup_{k\in[n^\iota,(n+1)^\iota)}E_k^c\right)
\leq
\mathbf{P}\left({\bar F}_{(n+1)^\iota,n^\iota}>0\right)
\leq
C n^{-(1+\delta)}.
\end{equation}

Therefore, by \eqref{eq:4.6} and \eqref{eq:4.12}, for all sufficiently large $N$,
$$
\begin{aligned}
\mathbf{P}\left(\bigcup_{n\ge N}(D_n\cap E_n)^c\right)
\leq&
\sum_{n\ge N}\mathbf{P}(D_n^c)+\sum_{r\ge \lfloor N^{1/\iota}\rfloor}
\mathbf{P}\left(\bigcup_{k\in[r^\iota,(r+1)^\iota)}E_k^c\right)\\
\leq&
C\sum_{n\ge N} n^{-(1+\varepsilon/2)} +C\sum_{r\ge \lfloor N^{1/\iota}\rfloor}r^{-(1+\delta)}
\leq
C N^{-\varepsilon/2} +C N^{-\delta/\iota}
\leq
C N^{-((\varepsilon/2)\wedge(\delta/\iota))}.
\end{aligned}
$$
Letting $N\to\infty$ gives
$\mathbf{P}\left(\bigcap_{N\ge1}\bigcup_{n\geq N}(D_n\cap E_n)^c\right)=0.$
Equivalently,
$
\mathbf{P}\left(\liminf\limits_{n\to\infty}(D_n\cap E_n)\right)=1.
$ The proof is complete.
\end{proof}

\begin{proof}[\bf Proof of Lemma \ref{lem:2.6}]
By Lemma~\ref{lem:2.4}, there exists a null event $\mathcal{N}_1$ such that, for every $\omega\in\Omega\setminus\mathcal{N}_1$, there exists a finite $N'=N'(\omega)\ge 2$ satisfying
$$
\omega\in D_n\cap E_n,~~ \forall n>N'.
$$
Moreover, since $\xi^*(n)\to\infty$ almost surely, there exists a null event $\mathcal{N}_2$  such that, for every $\omega\in\Omega\setminus\mathcal{N}_2$,
$$
\mathbf{1}_{C_m^1}(\omega)=1,~ \forall m\geq1.
$$
Set $\mathcal{N}:=\mathcal{N}_1\cup\mathcal{N}_2$. In what follows, we fix
$\omega\in\Omega\setminus\mathcal{N}$.

If $n>N'$ and $\xi^*(n,\omega)=m$, then the definition of $D_n$ gives
\begin{equation}\label{eq:4.13}
\exp\left(\frac{m}{(2+\varepsilon/2)\lambda}\right)
\leq n \leq
\exp\left(\frac{m}{(1-\varepsilon/2)\lambda}\right).
\end{equation}
Choose $M'=M'(\varepsilon)$ such that, for all $m\ge M'$,
\begin{equation}\label{eq:4.14}
(1-\varepsilon)\lambda \log\left\{\exp\left(\frac{m+1}{(1-\varepsilon/2)\lambda}\right) +\frac{m^2}{16}\right\}
\leq m.
\end{equation}
Set
\begin{equation}\label{eq:4.15}
M:=M'\vee \xi^*(N')\vee(\lambda\log N')^2.
\end{equation}

We first prove that, for all $m>M$ and $k\geq1$,
$\omega\in C_m^k$ implies $\omega\in \widetilde{A}_m^{k+1}.$
The case $k=1$ is immediate from the definition of $\widetilde{A}_m^2$. Let $k\geq 2$. Suppose, on the contrary, $\omega\in C_m^k$ but $ \omega\notin \widetilde{A}_m^{k+1}$.
Since $C_m^k=\{T_m^k<T_{m+1}^1\}$, the maximal local time at time $T_m^k(\omega)$ is $m$, and
$L_m^1(\omega),\ldots,L_m^k(\omega)$ are favorite sites.
By the definition of $\widetilde{A}_m^{k+1}$, we know that 
for some $k'\in\{1,\ldots,k-1\}$, the walk visits $L_m^{k'}(\omega)$ during $\left(T_m^k(\omega),T_m^k(\omega)+\frac{m^2}{16}
\right].$
Put
$$
\bar{T}_m^k(\omega):=T_m^k(\omega)+\lfloor m^2/16\rfloor.
$$
Then there exist $i<j\leq\bar{T}_m^k(\omega)$ such that $S_i(\omega),S_j(\omega)\in\{L_m^k(\omega),L_m^{k'}(\omega)\},~S_i(\omega)\neq S_j(\omega),~j-i\leq\frac{m^2}{16},$ and $S_i(\omega),S_j(\omega)\notin S_{(i,j)}(\omega),~\xi(S_i(\omega),\bar{T}_m^k(\omega))\wedge\xi(S_j(\omega),\bar{T}_m^k(\omega))\geq m.$
Since $m>M\geq \xi^*(N')(\omega)$, we have $T_m^1(\omega)>N'$. Applying \eqref{eq:4.13} at $T_m^1(\omega)$ and $T_{m+1}^1(\omega)$, we get
\begin{equation}\label{eq:4.16}
\exp\left(\frac{m}{4\lambda}\right)
\leq
T_m^1(\omega)<T_m^k(\omega)<T_{m+1}^1(\omega)
\leq
\exp\left(\frac{m+1}{(1-\varepsilon/2)\lambda}\right).
\end{equation}
Thus $\bar T_m^k(\omega)>N'$.
Moreover, by \eqref{eq:4.14} and \eqref{eq:4.16}, for all sufficiently large $m$,
$$
m\geq(1-\varepsilon)\lambda\log\bar{T}_m^k(\omega),
~~
\frac{m^2}{16}<(\lambda\log\bar{T}_m^k(\omega))^2 .
$$
Therefore, at time $\bar{T}_m^k(\omega)$, the sites $L_m^k(\omega)$ and $L_m^{k'}(\omega)$ form a pair counted by $\bar{F}_{\bar{T}_m^k(\omega)}$. This contradicts $E_{\bar{T}_m^k(\omega)}=\{\bar{F}_{\bar{T}_m^k}(\omega)=0\}$.
This proves the claim.

We now prove that, for all $m>M$ and $k\geq2$, $\omega\in B_m^k$ implies $\omega \in C_m^k$.  Assume $\omega\in B_m^k$. Recall that, for $k\geq 1$,
$$
B_m^k=A_m^1\cap\cdots\cap A_m^k,
~~
\widetilde{B}_m^k=\widetilde{A}_m^1\cap\cdots\cap\widetilde{A}_m^k,
~~
C_m^k=B_m^k\cap\widetilde{B}_m^k .
$$
Equivalently, for $k\geq2$,
$$
C_m^k
=
C_m^{k-1}\cap A_m^k\cap\widetilde{A}_m^k .
$$
Since $\xi^*(n)\to\infty$, we have $\omega\in C_m^1$ for every $m\geq1$.
Let $m>M$. We prove by induction on $j$ that $\omega\in C_m^j$ for $1\leq j\leq k$. The case $j=1$ is immediate. If $\omega\in C_m^j$ for some  $1\leq j\leq k-1$, then $\omega\in A_m^{j+1}$ because $\omega\in B_m^k$, and the claim gives
$\omega\in\widetilde A_m^{j+1}$. Hence
$$
\omega\in C_m^j\cap A_m^{j+1}\cap\widetilde{A}_m^{j+1}=C_m^{j+1}.
$$
The induction gives $\omega\in C_m^k$, and hence
 $1_{B_m^k}(\omega)\leq  1_{C_m^k}(\omega),~m>M,~ k\geq2.$

 The reverse inequality $1_{B_m^k}(\omega)\geq  1_{C_m^k}(\omega)$ is obvious since
 $C_m^k=B_m^k\cap\widetilde{B}_m^k.$ Thus, for every $\omega\in\Omega\setminus\mathcal{N}$,
$$
\mathbf{1}_{C_m^k}(\omega)=\mathbf{1}_{B_m^k}(\omega),~~m>M,~~k\geq2.
$$
The proof is complete.
\end{proof}

\section{Proofs of Theorems \ref{thm:1.4} and  \ref{thm:1.5}}\label{sec:4}

To prove Theorems~\ref{thm:1.4} and~\ref{thm:1.5}, we first introduce a regeneration structure and prove several auxiliary results.

\subsection{Regeneration Structure }
For a  random variable $Y$ on a probability space $(\Omega, \mathcal{F}, \mathbf{P})$, we use
$$
\mathcal{L}_{\mathbf{P}}(Y):=\mathbf{P}\circ Y^{-1}
$$
to denote the distribution of $Y$ under $\mathbf{P}$. Set
\begin{align}\label{3.0}
\mathcal{A}_0:=\{S_j\geq0,\ \forall j\geq0\}.
\end{align}
Since $\mu\in (0, \infty)$, \cite[Lemma~2.1]{Vilkas2025} yields
\begin{align}\label{3.1}
p:=\mathbf{P}(\mathcal{A}_0)=\mathbf{P}\left(S_j\geq0,\ \forall j\geq0\right)>0.
\end{align}
Define
$\bar{\mathbf{P}}(\,\cdot\,):=\mathbf{P}(\,\cdot\mid\mathcal{A}_0),$
and let $\bar{\mathbf{E}}$ denote the corresponding expectation.

Define the strict ascending ladder epochs by
\begin{align}\label{3.1-a}
\sigma_0:=0,
~~
\sigma_{k+1}
:=
\inf\left\{n>\sigma_{k}: S_n>\max_{0\leq j<n}S_j \right\},~~ \forall k\geq 0.
\end{align}
Set
\begin{align}\label{3.1-b}
Z_k:=S_{\sigma_k},
~~
\bar{\sigma}_k:=\sigma_k- \sigma_{k-1},
~~
\bar{Z}_k:=Z_k-Z_{k-1},~~k\geq1.
\end{align}
For $k\geq1$, define the strict ascending ladder block by
\begin{align}\label{3.1-c}
\Lambda_k
:=
\left(\bar{\sigma}_k,\,\bar{Z}_k,\,  \left(S_{\sigma_{k-1}+j}-S_{\sigma_{k-1}}\right)_{0\leq j<\bar{\sigma}_k} \right).
\end{align}

\begin{lem}\label{lem:3.1}
Almost surely,
\begin{equation}\label{eq:3.1}
\sigma_k<\infty,~~\forall k\geq1.
\end{equation}
Moreover,
\begin{equation}\label{eq:3.2}
(\Lambda_k)_{k\geq1}~~\text{is i.i.d.}.
\end{equation}
In addition,
\begin{equation}\label{eq:3.3}
\mathbf{E}[\bar{\sigma}_1]<\infty,~~
 \mathbf{E}[\bar{Z}_1]=\mu\mathbf{E}[\bar{\sigma}_1]<\infty.
\end{equation}
\end{lem}

\begin{proof}[\bf Proof]
By the strong law of large numbers, $\frac{S_n}{n}\to\mu>0~\text{a.s.}$. Hence $S_n\to\infty$ a.s., which implies \eqref{eq:3.1}.

For $k\geq1$,  we have
$$
\bar{\sigma}_k
=
\inf\left\{n\geq1:S_{\sigma_{k-1}+n}-S_{\sigma_{k-1}}>0\right\},
~~
\bar{Z}_k=S_{\sigma_{k-1}+\bar{\sigma}_k}-S_{\sigma_{k-1}}.
$$
Thus $\Lambda_k$ is the first strict ascending ladder block of the process
$\left(S_{\sigma_{k-1}+n}-S_{\sigma_{k-1}}\right)_{n\geq0}.$
For any bounded measurable function $f$, the strong Markov property at
$\sigma_{k-1}$ gives
$$
\mathbf{E}\left[f(\Lambda_k)\mid\mathcal{F}_{\sigma_{k-1}}\right]
=\mathbf{E}[f(\Lambda_1)],\  \forall k\geq 1,~~\text{a.s.}.
$$
In particular,
$\mathbf{E}[f(\Lambda_k)]=\mathbf{E}[f(\Lambda_1)],\ \forall k\geq1,$
and thus
$\Lambda_k$ and $\Lambda_1$ have the same distribution.
Moreover,
$\sigma(\Lambda_1,\ldots,\Lambda_{k-1})\subseteq\mathcal{F}_{\sigma_{k-1}}.$
Hence, for any $n\geq1$ and bounded measurable functions $f_1,\ldots,f_n$,  we have
$$
\begin{aligned}
\mathbf{E}\left[\prod_{k=1}^n f_k(\Lambda_k)\right]
&=
\mathbf{E}\left[\prod_{k=1}^{n-1}f_k(\Lambda_k)
\mathbf{E}\left[f_n(\Lambda_n)\mid\mathcal{F}_{\sigma_{(n-1)}}\right]\right]
=
\mathbf{E}\left[\prod_{k=1}^{n-1}f_k(\Lambda_k)\right]
\mathbf{E}[f_n(\Lambda_1)]\\
&=
\prod_{k=1}^n\mathbf{E}[f_k(\Lambda_1)]
=
\prod_{k=1}^n\mathbf{E}[f_k(\Lambda_k)].
\end{aligned}
$$
Therefore, $(\Lambda_k)_{k\geq1}$ is an i.i.d. sequence, which implies \eqref{eq:3.2}.

Finally,
$$
\bar{\sigma}_1=\sigma_1:=\inf\{n\geq1:S_n>0\},
~~
\bar{Z}_1=S_{\bar{\sigma}_1}.
$$
Since $\mu\in (0,\infty)$, Gut~\cite[Theorem~3.3.1]{Gut2009} yields $\mathbf{E}[\bar{\sigma}_1]<\infty.$ By Wald's identity and Gut~\cite[Theorem~1.5.3]{Gut2009}, we get
$$
\mathbf{E}[\bar{Z}_1]
=\mathbf{E}[S_{\bar{\sigma}_1}]
=\mu\mathbf{E}[\bar{\sigma}_1]<\infty.
$$
Hence \eqref{eq:3.3} holds. The proof is complete.
\end{proof}

Recall that $\mathcal{A}_0=\{S_j\geq0,\ \forall j\geq0\}$ and
$p=\mathbf{P}(\mathcal{A}_0)>0$ (cf. \eqref{3.1}). For $k\geq1$, let
\begin{equation}\label{eq:3.3-a}
\mathcal{A}_k
:=
\left\{
S_{\sigma_k+\ell}\geq S_{\sigma_k},
\ \forall \ell\geq0
\right\}.
\end{equation}
For $k,r\geq1$, define
\begin{equation}\label{eq:3.3-b}
A_k:=\sum_{j=1}^k\mathbf{1}_{\mathcal{A}_j},
~~
a_r:=\inf\{k\geq1:A_k=r\}.
\end{equation}

By (\ref{3.1-b}) and (\ref{3.1-c}), we know that  for any $k\geq 1$,
\begin{align*}
\Lambda_k
=
\left(\sigma_k-\sigma_{k-1},\ S_{\sigma_k}-S_{\sigma_{k-1}},\ \left(S_{\sigma_{k-1}+j}-S_{\sigma_{k-1}}\right)_{0\leq j<\sigma_k-\sigma_{k-1}} \right).
\end{align*}
For convenience, write $\Lambda_k$ as $
(\Lambda_{k,1},\Lambda_{k,2},\Lambda_{k,3}),
$
where
$$
\Lambda_{k,1}=\sigma_k-\sigma_{k-1},\ \Lambda_{k,2}=S_{\sigma_k}-S_{\sigma_{k-1}},\ \Lambda_{k,3}=\left(S_{\sigma_{k-1}+j}-S_{\sigma_{k-1}}\right)_{0\leq j<\sigma_k-\sigma_{k-1}}.
$$
For any real-valued random variable $a$ and any  positive integer $k$, we define
$$
a+\Lambda_{k, 3}:=\left(a+S_{\sigma_{k-1}+j}-S_{\sigma_{k-1}}\right)_{0\leq j<\sigma_k-\sigma_{k-1}}.
$$
 For $r\geq0$, define
$$\theta^r(\Lambda_1,\Lambda_2,\Lambda_3,\ldots):=(\Lambda_{r+1},\Lambda_{r+2},\Lambda_{r+3},\ldots),$$
and
$$
\Phi(\Lambda_{r+1},\Lambda_{r+2},\Lambda_{r+3},\ldots):=\left(\Lambda_{r+1,3}, \Lambda_{r+1,2}+\Lambda_{r+2,3},\sum_{j=r+1}^{r+2}\Lambda_{j,2}+\Lambda_{r+3,3},
\sum_{j=r+1}^{r+3}\Lambda_{j,2}+\Lambda_{r+4,3},\ldots \right).
$$
Then one can check that
\begin{align*}
&\Phi(\Lambda_1,\Lambda_2,\Lambda_3,\ldots)=
(0, S_1,S_2,\ldots),\\
&\Phi(\Lambda_{r+1},\Lambda_{r+2},\Lambda_{r+3},\ldots)=(0, S_{\sigma_r+1}-S_{\sigma_r},S_{\sigma_r+2}-S_{\sigma_r},\ldots),\ \forall r\geq 1.
\end{align*}
It follows that for any $r\geq 0$,
\begin{align*}
\mathbf{1}_{\mathcal{A}_r}
=(\mathbf{1}_{\mathcal{A}_0}\circ\Phi)
\left(\theta^r(\Lambda_1,\Lambda_2,\Lambda_3,\ldots)\right).
\end{align*}

By Lemma~\ref{lem:3.1}, the shift $\theta$ is measure preserving and ergodic, so Birkhoff's ergodic theorem \cite[Theorem~6.2.1]{D2019} gives
\begin{equation}\label{eq:3.4}
\begin{aligned}
\frac{A_k}{k}
=&
\frac{1}{k}\sum_{j=1}^k\mathbf{1}_{\mathcal{A}_j}
=
\frac{1}{k}\sum_{j=1}^k
(\mathbf{1}_{\mathcal{A}_0}\circ\Phi)\bigl(\theta^j(\Lambda_1,\Lambda_2,\ldots)\bigr)\\
\to&
\mathbf{E}\left[(\mathbf{1}_{\mathcal{A}_0}\circ\Phi)(\Lambda_1,\Lambda_2,\ldots)\right]
=
\mathbf{E}\left[\mathbf{1}_{\mathcal{A}_0}\right]
=
\mathbf{P}(\mathcal{A}_0)
=
p,~~\text{a.s.}.
\end{aligned}
\end{equation}
Since $p>0$, \eqref{eq:3.4} gives $A_k\to\infty$ almost surely. Hence, almost surely,
$a_r<\infty,\ \forall r\geq1, \lim_{r\to\infty}a_r=\infty.$
By the fact that $A_{a_r}=r$ and  \eqref{eq:3.4}, we have
\begin{equation}\label{eq:3.5}
\frac{r}{a_r}=\frac{A_{a_r}}{a_r}\to p,
\quad \frac{a_r}{r}\to\frac{1}{p},~~\text{a.s.}.
\end{equation}

Following Sznitman and Zerner \cite{SZ1999}, we introduce the regeneration epochs and prove the following two lemmas by adapting the arguments of their Proposition~1.3 and Theorem~1.4.
Set $\tau_0:=0$ and, for $r\geq1$, define
\begin{equation}\label{eq:3.6}
\tau_r:=
\inf\left\{n>\tau_{r-1}: S_n>\max_{0\leq j<n}S_j,\ S_{n+i}\geq S_n,\ \forall i\geq0 \right\}.
\end{equation}
Then by (\ref{3.1-a}), (\ref{eq:3.3-a}) and (\ref{eq:3.3-b}), we have  that  $\tau_r=\sigma_{a_r}, \forall r\geq 1$.
We call $\tau_r$ the $r$th regeneration epoch. Equivalently, it is the $r$th strict ascending ladder epoch after which the random walk never falls below its ladder height. Hence
$$
S_n<S_{\tau_r},~\forall~n\in[0, \tau_r)
~~\text{and}~~
S_n\geq S_{\tau_r},~\forall~n\geq\tau_r.
$$
$\tau_r$ is not a stopping time with respect to the natural filtration. Moreover,  by $\tau_r=\sigma_{a_r}$ and (\ref{3.1-b}), we have
$$
S_{\tau_r}=Z_{a_r},~~r\geq1,
~~\text{and}~~
0<\tau_1<\tau_2<\cdots<\infty,~~\text{a.s.}.
$$

For $n\geq1$, set
$$
B_n:=\{n\}\times\mathbb{Z}\times\mathbb{Z}^n,
~~
\mathcal{B}:=\bigcup_{n\geq1}B_n,
$$
and equip $\mathcal{B}$ with the discrete $\sigma$-field
$\sigma(\mathcal{B}):=2^{\mathcal{B}}.$
Let
$$
\mathcal{C}:=\mathbb{Z}^{\mathbb{Z}_+},
$$
and equip $\mathcal{C}$ with the product $\sigma$-field
$\sigma(\mathcal{C}):=\bigotimes_{j=0}^{\infty}2^{\mathbb{Z}}.$
Finally, set
$$
\Gamma_1:=\left(\tau_1,\,S_{\tau_1},\,(S_j)_{0\leq j<\tau_1}\right).
$$

\begin{lem}\label{lem:3.2}
For any bounded measurable functions $f:\mathcal{B}\to\mathbf{R}$ and $g:\mathcal{C}\to\mathbf{R}$,
  \begin{equation}\label{eq:3.7}
\mathbf{E}\left[f(\Gamma_1) g(\left(S_{\tau_1+i}-S_{\tau_1}\right)_{i\geq0})\right]
=
\mathbf{E}[f(\Gamma_1)]\bar{\mathbf{E}}\left[g((S_j)_{j\geq0})\right],
\end{equation}
and
  \begin{equation}\label{eq:3.8}
\bar{\mathbf{E}}\left[f(\Gamma_1)g(\left(S_{\tau_1+i}-S_{\tau_1}\right)_{i\geq0})\right]
=
\bar{\mathbf{E}}[f(\Gamma_1)]\bar{\mathbf{E}}\left[g((S_j)_{j\geq0})\right].
\end{equation}
\end{lem}

\begin{proof}[\bf Proof]
For $n\geq1$, define
$$
B_n:=
\left\{S_n>\max_{0\leq j<n}S_j\right\}
\cap
\bigcap_{\ell=1}^{n-1}\left(\left\{S_\ell\leq\max_{0\leq j<\ell}S_j\right\}
\cup
\left\{\min_{\ell<j\leq n}S_j<S_\ell\right\}\right).
$$
Then $B_n\in\mathcal{F}_n$ and
$\{\tau_1=n\}=B_n\cap\left\{S_{n+j}-S_n\geq0,\ \forall j\geq0\right\}.$
For each deterministic $n$, $\left(S_{n+i}-S_n\right)_{i\geq0}$ is independent of $\mathcal{F}_n$ and has the same law as $(S_j)_{j\geq0}$. Hence
$$
\begin{aligned}
&\mathbf{E}\left[f(\Gamma_1)g(\left( S_{\tau_1+i}-S_{\tau_1}\right)_{i\geq0})\right]\\
&=
\sum_{n\geq1}\mathbf{E}\left[f\left(n,S_n,(S_j)_{0\leq j<n}\right)
\mathbf{1}_{B_n}
g(\left(S_{n+i}-S_n\right)_{i\geq0})
\mathbf{1}_{\{S_{n+j}-S_n\geq0,\ \forall j\geq0\}}\right]\\
&=
\sum_{n\geq1}\mathbf{E}\left[f\left(n,S_n,(S_j)_{0\leq j<n}\right)
\mathbf{1}_{B_n}
\mathbf{E}\left[g(\left(S_{n+i}-S_n\right)_{i\geq0})
\mathbf{1}_{\{S_{n+j}-S_n\geq0,\ \forall j\geq0\}} \,\middle|\, \mathcal{F}_n\right]\right]\\
&=
\sum_{n\geq1}\mathbf{E}\left[f\left(n,S_n,(S_j)_{0\leq j<n}\right) \mathbf{1}_{B_n}\right]
\mathbf{E}\left[g((S_j)_{j\geq0})\mathbf{1}_{\{S_j\geq0,\ \forall j\geq0\}} \right]\\
&=
p\,\bar{\mathbf{E}}\left[g((S_j)_{j\geq0})\right]
\sum_{n\geq1}\mathbf{E}\left[f\left(n,S_n,(S_j)_{0\leq j<n}\right) \mathbf{1}_{B_n}\right].
\end{aligned}
$$
Taking $g\equiv1$ gives
$
\mathbf{E}[f(\Gamma_1)]
=
p\sum_{n\geq1}\mathbf{E}\left[f\left(n,S_n,(S_j)_{0\leq j<n}\right) \mathbf{1}_{B_n}\right].
$
Substituting this identity into the preceding display proves \eqref{eq:3.7}.

On $\{\tau_1=n\}$, we have  $S_{n+j}\geq S_n>0,~\forall j\geq0.$ Therefore,
$$
\mathcal{A}_0\cap\{\tau_1=n\}
=
\left\{\min_{0\leq j<n}S_j\geq0\right\}\cap\{\tau_1=n\}.
$$
Define $h:\mathcal{B}\to\{0,1\}$ by
$
h\left(n,x,(x_j)_{0\leq j<n}\right):=\mathbf{1}_{\{\min_{0\leq j<n}x_j\geq0\}}.
$
Since $\tau_1<\infty$ a.s., $\mathbf{1}_{\mathcal{A}_0}=h(\Gamma_1)~\text{a.s.}$.  Applying \eqref{eq:3.7} with $fh$ in place of $f$, we obtain
$$
\begin{aligned}
&\bar{\mathbf{E}}\left[f(\Gamma_1)g(\left( S_{\tau_1+i}-S_{\tau_1}\right)_{i\geq0})\right]\\
&=
\frac{1}{p}\mathbf{E}\left[f(\Gamma_1)g(\left(S_{\tau_1+i}-S_{\tau_1}\right)_{i\geq0})
\mathbf{1}_{\mathcal{A}_0}\right]
=
\frac{1}{p}\mathbf{E}\left[f(\Gamma_1)h(\Gamma_1) g(\left(S_{\tau_1+i}-S_{\tau_1}\right)_{i\geq0})\right]\\
&=
\frac{1}{p}\mathbf{E}\left[f(\Gamma_1)h(\Gamma_1)\right]
\bar{\mathbf{E}}\left[g\left((S_j)_{j\geq0}\right)\right]
=
\frac{1}{p}\mathbf{E}\left[f(\Gamma_1)\mathbf{1}_{\mathcal{A}_0}\right]
\bar{\mathbf{E}}\left[g\left((S_j)_{j\geq0}\right)\right]
=
\bar{\mathbf{E}}\left[f(\Gamma_1)\right]
\bar{\mathbf{E}}\left[g\left((S_j)_{j\geq0}\right)\right].
\end{aligned}
$$
This proves \eqref{eq:3.8}.
\end{proof}

\medskip

For $r\geq1$, set
$$
\Pi_r:=
\left(\tau_{r+1}-\tau_r,\,S_{\tau_{r+1}}-S_{\tau_r},\,
\left(S_{\tau_r+j}-S_{\tau_r}\right)_{0\leq j<\tau_{r+1}-\tau_r}\right).
$$
\begin{pro}\label{pro:3.3}
The sequence
\begin{equation}\label{eq:3.9}
(\Pi_r)_{r\geq1}\quad\text{is i.i.d.}.
\end{equation}
Moreover, the first two coordinates of $\Pi_1$ are integrable, and
\begin{equation}\label{eq:3.10}
\alpha:=\mathbf{E}[\tau_2-\tau_1]=\frac{\mathbf{E}[\bar{\sigma}_1]}{p},
~~
\beta:=\mathbf{E}[S_{\tau_2}-S_{\tau_1}]=\frac{\mathbf{E}[\bar{Z}_1]}{p}.
\end{equation}
In particular, $\alpha,\beta\in(0,\infty)$ and
\begin{equation}\label{eq:3.11}
\frac{\tau_r}{r}\to\alpha,
~~~
\frac{S_{\tau_r}}{r}\to\beta,~~ \text{a.s.},
~~~
\frac{\beta}{\alpha}=\mu.
\end{equation}
\end{pro}

\begin{proof}[\bf Proof]
For a path $\mathbf{x}=(x_j)_{j\geq0}\in\mathcal{C},$
define its first regeneration epoch $\tau_1(\mathbf{x})$ by the same formula as that for $(S_j)_{j\geq0}$. On $\{\tau_1(\mathbf{x})<\infty\}$, set
$$
\Psi(\mathbf{x})
:=
\left(\left(\tau_1(\mathbf{x}),x_{\tau_1(\mathbf{x})}-x_0,\left(x_j-x_0\right)_{0\leq j<\tau_1(\mathbf{x})}\right), \left(x_{\tau_1(\mathbf{x})+j}-x_{\tau_1(\mathbf{x})}\right)_{j\geq0}\right).
$$
On $\{\tau_1(\mathbf{x})=\infty\}$, assign $\Phi(\mathbf{x})$ an arbitrary fixed value in $\mathcal{B}\times\mathcal{C}$.
The map $\Psi: (\mathcal{C}, \sigma(\mathcal{C})) \to(\mathcal{B}\times\mathcal{C}, \sigma(\mathcal{B})\otimes \sigma(\mathcal{C}))$ is measurable.

For the random walk $(S_n)_{n\geq0}$,
\begin{equation}\label{eq:3.12}
\Psi\bigl((S_j)_{j\geq0}\bigr)
=\left(\Gamma_1,\left(S_{\tau_1+j}-S_{\tau_1}\right)_{j\geq0}\right),
~~\mathbf{P}\text{-a.s.}.
\end{equation}
The time shift and spatial translation invariance, together with
$S_n<S_{\tau_r}$ for $0\leq n<\tau_r$ and $S_n\geq S_{\tau_r}$ for $n\geq\tau_r$,
shows that $\tau_{r+1}-\tau_r$ is the first regeneration epoch of
$\left(S_{\tau_r+j}-S_{\tau_r}\right)_{j\geq0}$. Therefore,  for any $r\geq 1$,
\begin{equation}\label{eq:3.13}
\Psi
\left(\left(S_{\tau_r+j}-S_{\tau_r}\right)_{j\geq0}\right)
=
\left(\Pi_r, \left(S_{\tau_{r+1}+j}-S_{\tau_{r+1}}\right)_{j\geq0} \right),~~\mathbf{P}\text{-a.s.}.
\end{equation}

For $r\geq1$, let $\mathcal{D}_r:=\sigma(\Pi_1,\ldots,\Pi_r),$
with $\mathcal{D}_0$ the trivial $\sigma$-field. We prove by induction that, for every $r\geq1$,
\begin{equation}\label{eq:3.14}
\mathbf{E}\left[ Hh\left( \left(S_{\tau_r+j}-S_{\tau_r}\right)_{j\geq0}\right) \right]
=
\mathbf{E}[H]\, \bar{\mathbf{E}}\left[ h\bigl((S_j)_{j\geq0}\bigr)\right]
\end{equation}
for any bounded $\mathcal{D}_{r-1}$-measurable random variable $H$ and any bounded measurable function $h:\mathcal{C}\to\mathbf{R}$.

For $r=1$, $\mathcal{D}_0$ is trivial. Taking $f\equiv1$ in
\eqref{eq:3.7} gives
$\mathbf{E}\left[h\left(\left(S_{\tau_1+j}-S_{\tau_1}\right)_{j\geq0}\right) \right]
=
\bar{\mathbf{E}}\left[h\bigl((S_j)_{j\geq0}\bigr)\right],$
which proves \eqref{eq:3.14} for $r=1$. Suppose that \eqref{eq:3.14} holds for some $r\geq1$. Let $H$ be a bounded $\mathcal{D}_{r-1}$-measurable random variable, and let
$f:\mathcal{B}\to\mathbf{R}$ and $g:\mathcal{C}\to\mathbf{R}$ be bounded measurable functions.
 Write $\Psi=(\Psi_1,\Psi_2)$ and define
$$
h(\mathbf{x})
:=
f\left(\Psi_1(\mathbf{x})\right)g\left(\Psi_2(\mathbf{x})\right), ~~\mathbf{x}\in\mathcal{C}.
$$
 Since $\Psi$ is measurable, 
$h$ is bounded and measurable. Hence, by \eqref{eq:3.12}, \eqref{eq:3.13}, the induction hypothesis, and \eqref{eq:3.8},  we have
\begin{equation}\label{eq:3.15}
\begin{aligned}
\mathbf{E}\left[ Hf(\Pi_r)g\left(\left(S_{\tau_{r+1}+j}-S_{\tau_{r+1}}\right)_{j\geq0}\right)\right]
&=
\mathbf{E}\left[Hh\left(\left(S_{\tau_r+j}-S_{\tau_r}\right)_{j\geq0}\right)\right]
=
\mathbf{E}[H]\,\bar{\mathbf{E}}\left[h\bigl((S_j)_{j\geq0}\bigr)\right]\\
&=
\mathbf{E}[H]\,\bar{\mathbf{E}}\left[f(\Gamma_1) g\left(\left(S_{\tau_1+j}-S_{\tau_1}\right)_{j\geq0}\right)\right]\\
&=
\mathbf{E}[H]\,\bar{\mathbf{E}}\left[f(\Gamma_1)\right] \bar{\mathbf{E}}\left[g\bigl((S_j)_{j\geq0}\bigr)\right].
\end{aligned}
\end{equation}
Taking $g\equiv1$ in \eqref{eq:3.15} gives
\begin{equation}\label{eq:3.16}
\mathbf{E}\left[Hf(\Pi_r)\right]
=
\mathbf{E}[H]\,\bar{\mathbf{E}}\left[f(\Gamma_1)\right].
\end{equation}
Combining \eqref{eq:3.15} and \eqref{eq:3.16}, we obtain
\begin{equation}\label{eq:3.17}
\mathbf{E}\left[Hf(\Pi_r) g\left(\left(S_{\tau_{r+1}+j}-S_{\tau_{r+1}}\right)_{j\geq0}\right)\right]
=
\mathbf{E}\left[Hf(\Pi_r)\right] \bar{\mathbf{E}}\left[g\bigl((S_j)_{j\geq0}\bigr)\right].
\end{equation}

Fix a bounded measurable function $g:\mathcal{C}\to\mathbf{R}$, and let $\mathcal{H}_g$ be the class of bounded $\mathcal{D}_r$-measurable random variables $K$ satisfying
$$
\mathbf{E}\left[K g\left(\left(S_{\tau_{r+1}+j}-S_{\tau_{r+1}}\right)_{j\geq0}\right)\right]
=
\mathbf{E}[K]\,\bar{\mathbf{E}}\left[g\bigl((S_j)_{j\geq0}\bigr)\right].
$$
By \eqref{eq:3.17}, $\mathcal{H}_g$ contains all random variables $Hf(\Pi_r)$, where $H$ is bounded and $\mathcal{D}_{r-1}$-measurable and $f$ is bounded and measurable.
These random variables generate
$\mathcal{D}_r=\sigma(\mathcal{D}_{r-1},\Pi_r).$
Since $\mathcal{H}_g$ is a vector space closed under bounded monotone convergence, the functional monotone class theorem implies that it contains every bounded $\mathcal{D}_r$-measurable random variable. Thus \eqref{eq:3.14} holds with $r+1$ in place of $r$. This completes the induction.

Taking $H\equiv1$ in \eqref{eq:3.16}, we obtain
\begin{equation}\label{eq:3.18}
\mathcal{L}_{\mathbf{P}}(\Pi_r)=\mathcal{L}_{\bar{\mathbf{P}}}(\Gamma_1),~~ \forall r\geq1.
\end{equation}
Substituting \eqref{eq:3.18} into \eqref{eq:3.16} gives
$\mathbf{E}\left[Hf(\Pi_r)\right]=\mathbf{E}[H]\mathbf{E}\left[f(\Pi_r)\right].$
Hence $\Pi_r$ is independent of $\mathcal{D}_{r-1}$ for every $r\geq1$. Let $m\geq1$, and let $f_1,\ldots,f_m:\mathcal{B}\to\mathbf{R}$ be bounded measurable functions. Since $\prod_{r=1}^{m-1}f_r(\Pi_r)$  is $\mathcal{D}_{m-1}$-measurable, \eqref{eq:3.16} and \eqref{eq:3.18} give
$$
\mathbf{E}\left[\prod_{r=1}^{m}f_r(\Pi_r)\right]
=
\mathbf{E}\left[\prod_{r=1}^{m-1}f_r(\Pi_r)\right] \bar{\mathbf{E}}\left[f_m(\Gamma_1)\right]
=
\prod_{r=1}^{m}\bar{\mathbf{E}}\left[f_r(\Gamma_1)\right]
=
\prod_{r=1}^{m}\mathbf{E}\left[f_r(\Pi_r)\right].
$$
Therefore, $(\Pi_r)_{r\geq1}$ is independent and identically distributed, which proves \eqref{eq:3.9}.

Recall that $\tau_r=\sigma_{a_r}, S_{\tau_r}=Z_{a_r}$, $\bar{\sigma}_k=\sigma_k-\sigma_{k-1}$ and $\bar{Z}_k=Z_k-Z_{k-1}$. By \eqref{eq:3.5}, \eqref{eq:3.2}, and the strong law of large numbers, we have
\begin{equation}\label{eq:3.19}
\frac{1}{r}\sum_{j=1}^r(\tau_{j+1}-\tau_j)
=
\frac{1}{r}\sum_{k=a_1+1}^{a_{r+1}}\bar{\sigma}_k
=
\frac{a_{r+1}}{r}\left(\frac{1}{a_{r+1}}\sum_{k=1}^{a_{r+1}}\bar{\sigma}_k\right)
-\frac{1}{r}\sum_{k=1}^{a_1}\bar{\sigma}_k
\to\frac{\mathbf{E}[\bar{\sigma}_1]}{p}~~\text{a.s.},
\end{equation}
and
\begin{equation}\label{eq:3.20}
\frac{1}{r}\sum_{j=1}^r\left(S_{\tau_{j+1}}-S_{\tau_j}\right)
=
\frac{1}{r}\sum_{k=a_1+1}^{a_{r+1}}\bar{Z}_k
=
\frac{a_{r+1}}{r}\left(\frac{1}{a_{r+1}}\sum_{k=1}^{a_{r+1}}\bar{Z}_k\right)
-\frac{1}{r}\sum_{k=1}^{a_1}\bar{Z}_k
\to\frac{\mathbf{E}[\bar{Z}_1]}{p}~~\text{a.s.}.
\end{equation}
Fix $M\geq1$. By \eqref{eq:3.9} and the strong law of large numbers, we get that almost surely,
$$
\begin{aligned}
\frac{1}{r}\sum_{j=1}^r [(\tau_{j+1}-\tau_j)\wedge M]
\to \mathbf{E}[(\tau_2-\tau_1)\wedge M],~~
\frac{1}{r}\sum_{j=1}^r [(S_{\tau_{j+1}}-S_{\tau_j})\wedge M]
\to \mathbf{E}[(S_{\tau_2}-S_{\tau_1})\wedge M].
\end{aligned}
$$
Since both coordinates are nonnegative, comparison with \eqref{eq:3.19} and \eqref{eq:3.20} gives
$$
\mathbf{E}[(\tau_2-\tau_1)\wedge M] \leq\frac{\mathbf{E}[\bar{\sigma}_1]}{p},
~~
\mathbf{E}[(S_{\tau_2}-S_{\tau_1})\wedge M] \leq\frac{\mathbf{E}[\bar{Z}_1]}{p}.
$$
Letting $M\to\infty$ and applying the monotone convergence theorem, we conclude that $\tau_2-\tau_1$ and $S_{\tau_2}-S_{\tau_1}$ are integrable.

The strong law of large numbers for the i.i.d. sequence $(\Pi_r)_{r\geq1}$ now yields
$$
\frac{1}{r}\sum_{j=1}^r (\tau_{j+1}-\tau_j)
\to \mathbf{E}[\tau_2-\tau_1] ,~~
\frac{1}{r}\sum_{j=1}^r (S_{\tau_{j+1}}-S_{\tau_j})
\to \mathbf{E}[S_{\tau_2}-S_{\tau_1}],~~\text{a.s.}.
$$
Comparing this limit with \eqref{eq:3.19}, by  \eqref{eq:3.10} we obtain
$\mathbf{E}[\tau_2-\tau_1]=\frac{\mathbf{E}[\bar{\sigma}_1]}{p},
~\mathbf{E}[S_{\tau_2}-S_{\tau_1}]=\frac{\mathbf{E}[\bar{Z}_1]}{p}.$
In particular, $\alpha,\beta\in(0,\infty)$. Finally,
$$
\frac{\tau_r}{r}
=
\frac{\tau_1}{r}+\frac{1}{r}\sum_{j=1}^{r-1}(\tau_{j+1}-\tau_j)\to\alpha,
~~~
\frac{S_{\tau_r}}{r}
=
\frac{S_{\tau_1}}{r}+\frac{1}{r}\sum_{j=1}^{r-1}(S_{\tau_{j+1}}-S_{\tau_j})
\to\beta,~~\text{a.s.}.
$$
Moreover, by \eqref{eq:3.3} and \eqref{eq:3.10},
$\frac{\beta}{\alpha}
=
\frac{\mathbf{E}[\bar{Z}_1]}{\mathbf{E}[\bar{\sigma}_1]}=\mu.$
This proves \eqref{eq:3.11}. The proof is complete.
\end{proof}

\subsection{Block local time}

For $r\geq1$, define $\mathcal{E}_r:=[S_{\tau_r},S_{\tau_{r+1}})\cap\mathbb{Z}.$
For $x\in\mathcal{E}_r$, set
$$
\zeta_r(x):=\sum_{j=\tau_r}^{\tau_{r+1}-1}\mathbf{1}_{\{S_j=x\}},
~~
\zeta_r^*:=\max_{x\in\mathcal{E}_r}\zeta_r(x)
~~\text{and}~~
e_r:=\min\left\{x\in\mathcal{E}_r:\zeta_r(x)=\zeta_r^*\right\}.
$$

\begin{lem}\label{lem:3.4}
For any $r\geq1$ and $x\in\mathcal{E}_r$, $\xi(x,\infty)=\zeta_r(x).$ Moreover,
\begin{equation*}
\left(\zeta_r^*,e_r-S_{\tau_r}\right)_{r\geq1}\quad\text{is i.i.d.}.
\end{equation*}
\end{lem}

\begin{proof}[\bf Proof]
By the definitions of $\tau_r$ and  $\tau_{r+1}$ (cf. (\ref{eq:3.6})), we know that
$S_j<S_{\tau_r}, \forall\, 0\leq j<\tau_r$ and $S_j\geq S_{\tau_{r+1}}, \forall \, j\geq\tau_{r+1}.$
For any $x\in\mathcal{E}_r$, we have $S_{\tau_r}\leq x<S_{\tau_{r+1}}.$
Consequently,  if $x\in\mathcal{E}_r$, then for any $j<\tau_r$ or $j\geq\tau_{r+1}$, it holds that $\mathbf{1}_{\{S_j=x\}}=0,$ and thus
$$
\xi(x,\infty)
=
\sum_{j=\tau_r}^{\tau_{r+1}-1}\mathbf{1}_{\{S_j=x\}}
=
\zeta_r(x).
$$

By \eqref{eq:3.9}, $(\Pi_r)_{r\geq1}$ is i.i.d.. Furthermore,
$$
\zeta_r^*
=
\max_{0\leq y<S_{\tau_{r+1}}-S_{\tau_r}}
\sum_{j=0}^{\tau_{r+1}-\tau_r-1}\mathbf{1}_{\{S_{\tau_r+j}-S_{\tau_r}=y\}},
$$
and
$$
e_r-S_{\tau_r}
=
\min\left\{y~\left|~0\leq y<S_{\tau_{r+1}}-S_{\tau_r},
\sum_{j=0}^{\tau_{r+1}-\tau_r-1}\mathbf{1}_{\{S_{\tau_r+j}-S_{\tau_r}=y\}}
=\zeta_r^*\right.\right\}.
$$
Hence $\zeta_r^*$ and $e_r-S_{\tau_r}$ are two measurable functions of $\Pi_r$.
It follows that $\left(\zeta_r^*,e_r-S_{\tau_r}\right)_{r\geq1}$  is i.i.d.. The proof is complete.
\end{proof}

\medskip

Choose $a,b\in\mathbb{N}$ such that
$$\mathbf{P}(X_1=a)>0,~~\mathbf{P}(X_1=-b)>0.$$
Let
$$
d:=\gcd(a,b),~~u:=\frac{a}{d},~~v:=\frac{b}{d},
$$
and set
$$
\mathbf{y}
:=
(\underbrace{a,\ldots,a}_{v\text{ times}},\underbrace{-b,\ldots,-b}_{u\text{ times}})
=
(y_1,\ldots,y_t),
~~
t:=u+v.
$$
Define
$$
\eta:=\mathbf{P}\bigl((X_1,\ldots,X_t)=\mathbf{y}\bigr)>0,
~~
h_0:=0,
~~
h_q:=\sum_{j=1}^q y_j,\quad 1\leq q\leq t.
$$
Let
$$
\mathcal{H}_m
:=
\left\{\max_{x\in\mathbb{Z}}\sum_{j=0}^{\tau_1-1}\mathbf{1}_{\{S_j=x\}}\geq m \right\},
~~m\geq1.
$$

\begin{lem}\label{lem:3.5}
For every $m\geq1$,
\begin{equation}\label{eq:3.21}
\mathbf{P}(\zeta_1^*\geq m)=\bar{\mathbf{P}}(\mathcal{H}_m)>0.
\end{equation}
Moreover,
\begin{equation}\label{eq:3.22}
\mathbf{P}(\zeta_1^*\geq m+1)\geq\eta\mathbf{P}(\zeta_1^*\geq m).
\end{equation}
\end{lem}

\begin{proof}[\bf Proof]
By Lemma \ref{lem:3.2},  $\left(S_{\tau_1+i}-S_{\tau_1}\right)_{i\geq0}$ under $\mathbf{P}$ has the same distribution as $(S_j)_{j\geq0}$ under $\bar{\mathbf{P}}$. Moreover, $\tau_2-\tau_1$ is obtained from $\left(S_{\tau_1+i}-S_{\tau_1}\right)_{i\geq0}$ through the same measurable construction by which $\tau_1$ is obtained from $(S_j)_{j\geq0}$. Hence
\begin{equation}\label{eq:3.23}
\mathcal{L}_{\mathbf{P}}(\zeta_1^*)
=
\mathcal{L}_{\mathbf{P}}\left(\max_{x\in\mathbb{Z}} \sum_{j=0}^{\tau_2-\tau_1-1}\mathbf{1}_{\{S_{\tau_1+j}-S_{\tau_1}=x\}}\right)
=
\mathcal{L}_{\bar{\mathbf{P}}} \left(\max_{x\in\mathbb{Z}}\sum_{j=0}^{\tau_1-1}\mathbf{1}_{\{S_j=x\}}\right).
\end{equation}
Therefore, for all $m\geq1$,
$\mathbf{P}(\zeta_1^*\geq m)=\bar{\mathbf{P}}(\mathcal{H}_m).$

Recall that $\mathcal{A}_0$ is defined by \eqref{3.0}.
For $m\geq1$, let
$$
A_m
:=\left\{(X_{(q-1)t+1},\ldots,X_{qt})=\mathbf{y}, ~1\leq q\leq m-1\right\}.
$$
On $A_m$,
$S_{qt}=0,$ and $S_{(q-1)t+j}=h_j>0,  \forall \, 0\leq q\leq m-2,~ \forall \, 1\leq j<t.$
Hence the origin is visited at the times $0,t,\ldots,(m-1)t$. In particular, no regeneration epoch can occur before $(m-1)t$.
Consequently,
$A_m\cap\mathcal{A}_0\subseteq\mathcal{H}_m\cap\mathcal{A}_0.$
Since $S_{(m-1)t}=0$ on $A_m$,  by the Markov property and the  independence of the increments, we have
\begin{equation*}
\bar{\mathbf{P}}(\mathcal{H}_m)
=
\frac{\mathbf{P}(\mathcal{H}_m\cap\mathcal{A}_0)}{p}
\geq
\frac{\mathbf{P}(A_m\cap\mathcal{A}_0)}{p}
=
\frac{\mathbf{P}(A_m)\mathbf{P}(\mathcal{A}_0)}{p}
=
\mathbf{P}(A_m)
=
\eta^{m-1}>0.
\end{equation*}
This proves \eqref{eq:3.21}.

Fix $m\geq1$, and define
$$
T_m:=\inf\{n\geq0:\xi^*(n)=m\},~~H_n:=\mathcal{H}_m\cap\{T_m=n\}.
$$
Since $T_m<\tau_1$ on $\mathcal{H}_m$,
\begin{equation}\label{eq:3.24}
\mathcal{H}_m= \bigcup_{n\geq 0}H_n.
\end{equation}

For $n\geq0$, define
\begin{align}\label{eq:3.24-a}
\phi_n(x_1,x_2,\ldots):=(x_1,\ldots,x_n,y_1,\ldots,y_t,x_{n+1},x_{n+2},\ldots).
\end{align}
 For any $\mathbf{x}:=(x_1,x_2,\ldots)\in H_n$, we use $\omega_{\mathbf{x}}$ to denote the
path associated with $\mathbf{x}$. Then $S_0(\omega_{\mathbf{x}})=0, S_n(\omega_{\mathbf{x}})=\sum_{k=1}^nx_i$ for any $n\geq 1$.  By the definition of $H_n$, we know that $T_m(\omega_{\mathbf{x}})=n$ and $\xi^*(n)(\omega_{\mathbf{x}})=m=\xi(S_n(\omega_{\mathbf{x}}), n)(\omega_{\mathbf{x}})$, which together with (\ref{eq:3.24-a}) implies that $T_m(\omega_{\phi_n(\mathbf{x})})=n$.
Since
$\sum_{i=1}^t y_i=0$ and $\sum_{i=1}^j y_i\neq0$ for $1\leq j<t$, we know that for any $k\in \{n+1,n+2,\ldots,n+t\}$, $S_k(\omega_{\phi_n(\mathbf{x})})$ visits $S_n(\omega_{\phi_n(\mathbf{x})})$ only at the time $n+t$ and increasing the local time of the site $S_n(\omega_{\phi_n(\mathbf{x})})$ to $m+1$.
  By the definition of $H_n$, we know that $\tau_1(\omega_{\mathbf{x}})>n$. Then by the definitions of $\phi_n$ and $\tau_1$ (cf. (\ref{eq:3.24-a}) and (\ref{eq:3.6})), we have $\tau_1(\omega_{\phi_n(\mathbf{x}})>n+t$. Hence $\phi_n(\mathbf{x})\in \mathcal{H}_{m+1}\cap\{T_m=n\}$ and thus
\begin{equation}\label{eq:3.25}
\phi_n(H_n)\subseteq \mathcal{H}_{m+1}\cap\{T_m=n\}.
\end{equation}

Let $\Delta_n$  denote the transform which deletes the increments in positions $n+1,\ldots,n+t$. Since the inserted word stays above its initial level and has zero total displacement,  we have
$\mathbf{1}_{\mathcal{A}_0}(\omega)=\mathbf{1}_{\mathcal{A}_0}(\Delta_n\omega)$
on $\{(X_{n+1},\ldots,X_{n+t})=\mathbf{y}\}$. Hence, by the fact that $\{X_i\}_{i\geq 1}$ is i.i.d., we have
\begin{equation}\label{eq:3.26}
\begin{aligned}
&~~~~\bar{\mathbf{P}}(\phi_n(H_n))\\
&=
\frac{1}{p}\mathbf{P}\left(\phi_n(H_n)\cap\mathcal{A}_0\right)
=
\frac{1}{p}\mathbf{P}\left(\left\{(X_{n+1},\ldots,X_{n+t})=\mathbf{y},\ \Delta_n\omega\in H_n\right\}\cap\mathcal{A}_0\right)\\
&=
\frac{1}{p}\mathbf{P}\left((X_{n+1},\ldots,X_{n+t})=\mathbf{y},\
\Delta_n\omega\in H_n\cap\mathcal{A}_0\right)
=
\frac{1}{p}
\mathbf{E}\left[\mathbf{1}_{\{(X_{n+1},\ldots,X_{n+t})=\mathbf{y}\}}
\mathbf{1}_{\{\Delta_n\omega\in H_n\cap\mathcal{A}_0\}}\right]\\
&=
\frac{1}{p}\mathbf{E}\left[\mathbf{1}_{\{\Delta_n\omega\in H_n\cap\mathcal{A}_0\}} \mathbf{E}\left[\mathbf{1}_{\{(X_{n+1},\ldots,X_{n+t})=\mathbf{y}\}}
\,\middle|\,\sigma(X_1,\ldots,X_n,X_{n+t+1},X_{n+t+2},\ldots)\right] \right]\\
&=\frac{\eta}{p}\mathbf{P}\left(\Delta_n\omega\in H_n\cap\mathcal{A}_0\right)
=\frac{\eta}{p}\mathbf{P}(H_n\cap\mathcal{A}_0)
=\eta\bar{\mathbf{P}}(H_n).
\end{aligned}
\end{equation}
Note that $\phi_n(H_n)$, $n\ge 0$, are disjoint because they have distinct values of $T_m$.
Combining  \eqref{eq:3.21}, \eqref{eq:3.24}, \eqref{eq:3.25}, and \eqref{eq:3.26}, we obtain
$$
\mathbf{P}(\zeta_1^*\geq m+1) =\bar{\mathbf{P}}(\mathcal{H}_{m+1})
\geq
\sum_{n\geq0}\bar{\mathbf{P}}(\phi_n(H_n))
=
\eta\sum_{n\geq0}\bar{\mathbf{P}}(H_n)
=
\eta\bar{\mathbf{P}}(\mathcal{H}_m)=\eta \mathbf{P}(\zeta_1^*\geq m),
$$
 i.e., \eqref{eq:3.22} holds. The proof is complete.
\end{proof}

\subsection{Proof of Theorem \ref{thm:1.4} }\label{sec:5}

\begin{proof}[\bf Proof.]
Recall that $F_n$ is an arbitrary favorite site at time $n$.
Since $F_n$ has been visited by time $n$, $F_n\leq\max_{0\leq j\leq n}S_j.$ By the strong law of large numbers, we have
\begin{equation}\label{eq:3.27}
\limsup_{n\to\infty}\frac{F_n}{\mu n}
\leq \lim_{n\to\infty}\frac{\max_{0\leq j\leq n}S_j}{\mu n}=1,~~\text{a.s.}.
\end{equation}

We next prove the reverse inequality. By Lemmas~\ref{lem:3.4} and~\ref{lem:3.5}, $(\zeta_r^*)_{r\geq1}$ is i.i.d. and $\mathbf{P}(\zeta_1^*\geq m)>0,\ \forall m\geq1.$
Thus, for all $m\geq1$, the second Borel-Cantelli lemma gives $\mathbf{P}(\zeta_r^*\geq m\ \mathrm{i.o.})=1.$
Taking the intersection over $m\geq1$, we obtain
\begin{equation}\label{eq:3.28}
\max_{1\leq k\leq R}\zeta_k^*\to\infty,~~\text{a.s.}.
\end{equation}
Moreover, since $\tau_1$ is finite, we have that $\max_{x\in\mathbb{Z}}\xi(x,\tau_1-1)<\infty$ a.s..
Consequently, almost surely there exist infinitely many $R$'s
such that
$\zeta_R^*>
\max\left\{\max_{x\in\mathbb{Z}}\xi(x,\tau_1-1),\max_{1\leq k<R}\zeta_k^*\right\}.$
For any such $R$, Lemma~\ref{lem:3.4} gives, for every $x<S_{\tau_R}$,
$$
\xi(x,\tau_{R+1}-1)
\leq
\max\left\{\max_{y\in\mathbb{Z}}\xi(y,\tau_1-1),\max_{1\leq k<R}\zeta_k^*\right\}
<
\zeta_R^*,
$$
whereas
$\max_{S_{\tau_R}\leq x<S_{\tau_{R+1}}}\xi(x,\tau_{R+1}-1)=\zeta_R^*.$
Furthermore,
$\xi(x,\tau_{R+1}-1)=0,~x\geq S_{\tau_{R+1}}.$
It follows that
\begin{equation}\label{eq:3.29}
\mathcal{K}(\tau_{R+1}-1)\subseteq[S_{\tau_R},S_{\tau_{R+1}})\cap\mathbb{Z}.
\end{equation}
Hence
$$S_{\tau_R}\leq F_{\tau_{R+1}-1}<S_{\tau_{R+1}}.$$

By \eqref{eq:3.11},  we get that $\frac{S_{\tau_R}}{R}\to\beta$
and
$
\frac{\tau_{R+1}-1}{R}
=
\frac{R+1}{R}\frac{\tau_{R+1}}{R+1}-\frac{1}{R}\to\alpha\  \text{a.s.}.
$
Since $\beta/\alpha=\mu$,
$$
\frac{S_{\tau_R}}{\mu(\tau_{R+1}-1)}
=
\frac{S_{\tau_R}/R}{\mu(\tau_{R+1}-1)/R}
\to\frac{\beta}{\mu\alpha}
=1~~\text{a.s.}.
$$
Choose an increasing sequence $(R_j)_{j\geq1}$ of such $R$'s, using \eqref{eq:3.29},
$$
\limsup_{n\to\infty}\frac{F_n}{\mu n}
\geq
\limsup_{j\to\infty}\frac{F_{\tau_{(R_j+1)}-1}}{\mu(\tau_{(R_j+1)}-1)}
\geq
\lim_{j\to\infty}\frac{S_{\tau_{R_j}}}{\mu(\tau_{(R_j+1)}-1)}
=1~~\text{a.s.},
$$
which together with \eqref{eq:3.27} implies that for every choice $F_n\in\mathcal{K}(n)$,
$$
\limsup_{n\to\infty}\frac{F_n}{\mu n}=1,~~\text{a.s.}.
$$
 The proof is complete.
\end{proof}

\subsection{Proof of Theorem \ref{thm:1.5}}

To prove Theorem \ref{thm:1.5}, we need three lemmas. Their proofs are postponed to the next subsection.
For $r\geq1$, define
\begin{align}\label{3.4-1}
\mathcal{J}_r
:=
\left\{1\leq k\leq r:\zeta_k^*=\max_{1\leq j\leq r}\zeta_j^*\right\},
~~
I_r^-:=\min\mathcal{J}_r,
~~
I_r^+:=\max\mathcal{J}_r,
~~
J_r:=\#\mathcal{J}_r.
\end{align}
$I_r^-$ and $I_r^+$ are the first and last indices attaining the maximum block local time among the first $r$ regeneration blocks,  respectively.

\begin{lem}\label{lem:3.6}
There exist two positive constants $c_1$ and $c_2$ such that for all $r\geq2$,
\begin{equation}\label{eq:3.30}
\mathbf{P}(I_r^\pm=i)\leq\frac{c_2}{r},~~1\leq i\leq r,
\end{equation}
and
\begin{equation}\label{eq:3.31}
c_1\frac{l}{r}
\leq \mathbf{P}(I_r^\pm\leq l)
\leq c_2\frac{l}{r},~~1\leq l\leq\frac{r}{2}.
\end{equation}
Moreover,
\begin{equation}\label{eq:3.32}
I_r^-\to\infty \quad\text{and}\quad I_r^+\to\infty\ \text{a.s.}.
\end{equation}
\end{lem}

\begin{lem}\label{lem:3.7}
There exists a  constant $C>0$ such that for all integers
$1\leq r<s$ satisfying
$s\geq2r,~1\leq l\leq\frac{r}{2},~1\leq L\leq\frac{s}{2},$
we have
\begin{align}
&\mathbf{P}(I_r^+\leq l,\ I_s^+\leq L)
\leq C\frac{l}{s}+C\frac{l}{r}\frac{L}{s},\label{lem-3.7-a}\\
&\mathbf{P}(I_r^-\leq l,\ I_s^-\leq L)
\leq C\frac{l}{s}+C\frac{l}{r}\frac{L}{s}.\label{lem-3.7-b}
\end{align}
\end{lem}

\begin{lem}\label{pro:3.8}
For every non-decreasing function $b:\mathbb{N}\to(0,\infty)$,
\begin{eqnarray}\label{eq:3.33}
\liminf_{r\to\infty}\frac{b(r)I_r^\pm}{r}
=\left\{
\begin{array}{cl}
0\ a.s.& \mbox{if}~~
\sum_{r=2}^{\infty}\frac{1}{rb(r)}=\infty,\\
\infty\ a.s.&\mbox{if}~~
\sum_{r=2}^{\infty}\frac{1}{rb(r)}<\infty.
\end{array}
\right.
\end{eqnarray}
\end{lem}

\begin{proof}[\bf Proof of Theorem~\ref{thm:1.5}]
Let $a:\mathbb{N}\to(0,\infty)$ be non-decreasing. For every $c>0$, the monotonicity of $a(\cdot)$ and a grouping of terms give
\begin{equation}\label{eq:3.34}
\sum_{n=2}^{\infty}\frac{1}{na(n)}<\infty
\quad\Longleftrightarrow\quad
\sum_{n=2}^{\infty}\frac{1}{n a(\lfloor cn\rfloor\vee2)}<\infty.
\end{equation}

Choose
$0<\delta<\min\left\{1,\frac{1}{\alpha}\right\}$,  where $\alpha$ is defined in (\ref{eq:3.10}).
Set
$r_n:=\lfloor\delta n\rfloor\vee1.$
Then $r_n\to\infty$, $r_n/n\to\delta$, and \eqref{eq:3.11} gives
$$
\frac{\tau_{(r_n+1)}}{n}
=
\frac{\tau_{(r_n+1)}}{r_n+1}\frac{r_n+1}{n}
\to\alpha\delta<1,~~\text{a.s.}.
$$
Consequently, almost surely for all sufficiently large $n$,
\begin{equation}\label{eq:3.35}
\tau_{(r_n+1)}\leq n.
\end{equation}
By \eqref{eq:3.28} and $\max_{x\in\mathbb{Z}}\xi(x,\tau_1-1)<\infty~\text{a.s.}$, we get that almost surely for all sufficiently large $n$,
$
\zeta_{I_{r_n}^-}^*
=
\max_{1\leq k\leq r_n}\zeta_k^*
>
\max_{x\in\mathbb{Z}}\xi(x,\tau_1-1).
$
The definition of $I_{r_n}^-$ further gives $\zeta_k^*<\zeta_{I_{r_n}^-}^*,~1\leq k<I_{r_n}^-.$ It follows from \eqref{eq:3.35} and Lemma~\ref{lem:3.4} that, for
all $x<S_{\tau_{I_{r_n}^-}}$,
$$
\xi(x,n)
\leq
\max\left\{
\max_{y\in\mathbb{Z}}\xi(y,\tau_1-1),
\max_{1\leq k<I_{r_n}^-}\zeta_k^*
\right\}
<
\zeta_{I_{r_n}^-}^*
\leq
\xi^*(n).
$$
Thus almost surely for all
sufficiently large $n$,
\begin{equation}\label{eq:3.36}
F_n\geq S_{\tau_{I_{r_n}^-}}.
\end{equation}

Equation~\eqref{eq:3.32} implies that
$I_{r_n}^-\to\infty$ a.s.. Applying \eqref{eq:3.11} along the sequence $(I_{r_n}^-)_{n\geq1}$, we obtain
\begin{equation}\label{eq:3.37}
\frac{S_{\tau_{I_{r_n}^-}}}{I_{r_n}^-}\to\beta \ \text{a.s.}.
\end{equation}
Together with $r_n/n\to\delta$, \eqref{eq:3.37} implies that, almost surely for all sufficiently large $n$,
$
S_{\tau_{I_{r_n}^-}}\geq\frac{\beta}{2}I_{r_n}^-,
~
\frac{r_n}{n}\geq\frac{\delta}{2}.
$
By $r_n\leq n$ and $a(\cdot)$ is non-decreasing, \eqref{eq:3.36} yields
\begin{equation}\label{eq:3.38}
\frac{a(n)F_n}{n}
\geq
\frac{a(n)S_{\tau_{I_{r_n}^-}}}{n}
\geq
\frac{\beta}{2}\frac{a(n)I_{r_n}^-}{n}
\geq
\frac{\beta\delta}{4}\frac{a(r_n)I_{r_n}^-}{r_n}.
\end{equation}

Suppose first that
$$
\sum_{n=2}^{\infty}\frac{1}{na(n)}<\infty.
$$
Applying Lemma \ref{pro:3.8} to $I_r^-$ and using $r_n\to\infty$, \eqref{eq:3.38} yields
$$
\liminf_{n\to\infty}\frac{a(n)F_n}{n}
\geq
\frac{\beta\delta}{4}\liminf_{n\to\infty}\frac{a(r_n)I_{r_n}^-}{r_n}
\geq
\frac{\beta\delta}{4}\liminf_{r\to\infty}\frac{a(r)I_r^-}{r}
=
\infty,~~\text{a.s.}.
$$

Suppose next that
$$
\sum_{n=2}^{\infty}\frac{1}{na(n)}=\infty.
$$
At time $\tau_{r+1}-1$, the first $r$ regeneration blocks have been completed. Then Lemma~\ref{lem:3.4}  gives
$
\max_{x\in\mathcal{E}_k}\xi(x,\tau_{r+1}-1)=\zeta_k^*,~1\leq k\leq r.
$
Furthermore, \eqref{eq:3.28} shows that, almost surely for all sufficiently large $r$,
$\max_{1\leq k\leq r}\zeta_k^*>\max_{x\in\mathbb{Z}}\xi(x,\tau_1-1).$
For $x\geq S_{\tau_{r+1}}$, we also have
$\xi(x,\tau_{r+1}-1)=0.$
Hence every favorite site at time $\tau_{r+1}-1$ belongs to a block with the block index in $\mathcal{J}_r$ (cf. \eqref{3.4-1}). Thus, almost surely for all sufficiently large $r$,
\begin{equation*}
\mathcal{K}(\tau_{r+1}-1)
\subseteq
\bigcup_{k\in\mathcal{J}_r}\mathcal{E}_k
\subseteq
[S_{\tau_{I_r^-}},S_{\tau_{(I_r^+ +1)}})\cap\mathbb{Z}.
\end{equation*}
In particular,
\begin{align}\label{eq:3.39}
S_{\tau_{I_r^-}}\leq F_{\tau_{r+1}-1}<S_{\tau_{(I_r^+ +1)}}.
\end{align}

By~\eqref{eq:3.11},  we know that $\frac{\tau_{r+1}-1}{r}\to\alpha\ \text{a.s.}.$
Moreover, \eqref{eq:3.32} implies that $I_r^+\to\infty~ \text{a.s.}.$ Applying \eqref{eq:3.11} along $(I_r^++1)_{r\geq1}$, we get
$$
\frac{S_{\tau_{(I_r^+ +1)}}}{I_r^+}
=
\frac{S_{\tau_{(I_r^+ +1)}}}{I_r^++1}\frac{I_r^++1}{I_r^+}\to\beta~~ \text{a.s.}.
$$
Therefore, there exists a deterministic constant $C>1$ such that, for all sufficiently large $r$,
\begin{equation}\label{eq:3.40}
C^{-1}r\leq \tau_{r+1}-1 \leq Cr,
~~
S_{\tau_{(I_r^+ +1)}}\leq CI_r^+~~\text{a.s.}.
\end{equation}
Equation~\eqref{eq:3.34} yields
$\sum_{r=2}^{\infty}\frac{1}{r a(\lfloor Cr\rfloor\vee2)}=\infty.$
Since $r\mapsto a(\lfloor Cr\rfloor\vee2)$ is non-decreasing, Lemma \ref{pro:3.8}, applied to $I_r^+$, yields
$\liminf_{r\to\infty}\frac{a(\lfloor Cr\rfloor\vee2)I_r^+}{r}=0\ ~\text{a.s.}$.
Combining \eqref{eq:3.39}, \eqref{eq:3.40}, and the monotonicity of $a$, we obtain that almost surely for all sufficiently large $r$,
$$
\frac{a(\tau_{r+1}-1)F_{\tau_{r+1}-1}}{\tau_{r+1}-1}
\leq
\frac{a(\tau_{r+1}-1)S_{\tau_{(I_r^+ +1)}}}{\tau_{r+1}-1}
\leq
C^2\frac{a(\lfloor Cr\rfloor\vee2)I_r^+}{r}.
$$
Consequently,
$$
\liminf_{r\to\infty}\frac{a(\tau_{r+1}-1)F_{\tau_{r+1}-1}}{\tau_{r+1}-1}
\leq0~~ \text{a.s.}.
$$

On the other hand, \eqref{eq:3.36} and \eqref{eq:3.37} show that $F_n\geq0$ a.s. for all sufficiently large $n$. Since $\tau_{r+1}\to\infty~\text{a.s.}$, we have
$$
0
\leq
\liminf_{n\to\infty}\frac{a(n)F_n}{n}
\leq
\liminf_{r\to\infty}\frac{a(\tau_{r+1}-1)F_{\tau_{r+1}-1}}{\tau_{r+1}-1}
\leq0~~\text{a.s.}.
$$
Therefore, for every choice $F_n\in\mathcal{K}(n)$,
$$
\liminf_{n\to\infty}\frac{a(n)F_n}{n}=0~~ \text{a.s.}.
$$
The proof is complete.
\end{proof}

\subsection{Proofs of Lemmas 3.6, 3.7, 3.8}

\begin{proof}[\bf Proof of Lemma \ref{lem:3.6}]
Set
$$
p_m^*:=\mathbf{P}(\zeta_1^*\geq m),
~~
\bar p_m:=\mathbf{P}(\zeta_1^*=m)=p_m^*-p_{m+1}^*,~~m\geq1.
$$
By Lemmas~\ref{lem:3.4} and \ref{lem:3.5}, we know that $(\zeta_k^*)_{k\geq1}$ is i.i.d. and for any $m\geq1$, $1=p_1^*\geq p_m^*\downarrow0$, $p_{m+1}^*\geq\eta p_m^*>0$. By the i.i.d. property of $(\zeta_j^*)_{1\leq j\leq r}$, for every $1\leq k\leq r$,
$$
\begin{aligned}
\mathbf{P}(k\in\mathcal{J}_r)
&=
\sum_{m\geq1}\mathbf{P}\left(\zeta_k^*=m,\,\zeta_j^*\leq m,\ \forall \ 1\leq j\leq r,j\neq k\right)\\
&=\sum_{m\geq1} \mathbf{P}(\zeta_k^*=m) \prod_{1\leq j \leq r,~ j\neq k} \mathbf{P}(\zeta_j^*\leq m)
=
\sum_{m\geq1}\bar p_m(1-p_{m+1}^*)^{r-1}.
\end{aligned}
$$
Consequently, using $1-x\leq e^{-x}$,
\begin{align}\label{eq:3.41}
\mathbf{E}[J_r]
&=\sum_{k=1}^r\mathbf{P}(k\in\mathcal{J}_r)
=r\sum_{m\geq1}\bar{p}_m(1-p_{m+1}^*)^{r-1} \nonumber\\
&\leq
r\sum_{m\geq1}\bar{p}_m\exp\{-(r-1)p_{m+1}^*\}
\leq
r\sum_{m\geq1}(p_m^*-p_{m+1}^*)\exp\{-\eta(r-1)p_m^*\}\nonumber\\
&\leq
r\sum_{m\geq1}\int_{p_{m+1}^*}^{p_m^*}\exp\{-\eta(r-1)x\}\,dx
=
r\int_0^1\exp\{-\eta(r-1)x\}\,dx\nonumber\\
&=
\frac{r}{\eta(r-1)}\left(1-e^{-\eta(r-1)}\right)
\leq
\frac{2}{\eta},~~\forall r\geq2.
\end{align}
For $1\leq i \leq r$, $\{I_r^\pm=i\}\subseteq\{i\in\mathcal{J}_r\}.$
Hence
\eqref{eq:3.41} gives
\begin{align}\label{eq:3.41-a}
\mathbf{P}(I_r^\pm=i)
\leq
\mathbf{P}(i\in\mathcal{J}_r)
=
\frac{\mathbf{E}[J_r]}{r}
\leq
\frac{2}{\eta r},
\end{align}
which proves \eqref{eq:3.30}.

 Next, we prove \eqref{eq:3.31}. By \eqref{eq:3.41-a}, we have
\begin{align}\label{eq:3.41-b}
\mathbf{P}(I_r^\pm\leq l)
\leq
\sum_{i=1}^{l}\mathbf{P}(I_r^\pm=i)
\leq
\frac{2l}{\eta r}.
\end{align}
For any $1\leq q\leq r$, by the definitions in (\ref{3.4-1}), we know that
$$
\{J_r=q\}=\bigcup_{A\subset \{1,\ldots,r\}, \#A=q}\{\mathcal{J}_r=A\}.
$$
By the i.i.d. property of $(\zeta_j^*)_{1\leq j\leq r}$, we know that the probability $\mathbf{P}(\mathcal{J}_r=A)$ only depends on $\#A$.  Thus conditional on $\{J_r=q\}$, $\mathcal{J}_r$ is uniformly distributed over all $q$-element subsets of $\{1,\ldots,r\}$.
Therefore, for every $A\subseteq\{1,\ldots,r\}$ with $\#A=q$,
$
\mathbf{P}(\mathcal{J}_r=A\mid J_r=q)=\binom{r}{q}^{-1}.
$
If $q\leq r-l$, then
$$
\mathbf{P}(I_r^->l\mid  J_r=q)
=
\frac{\binom{r-l}{q}}{\binom{r}{q}}
=
\prod_{j=0}^{q-1}\frac{r-l-j}{r-j}
\leq
\frac{r-l}{r}.
$$
If $q>r-l$, this conditional probability is zero. Hence $\mathbf{P}(I_r^-\leq l\mid J_r=q)\geq\frac{l}{r},~1\leq q\leq r.$ Therefore
$$
\mathbf{P}(I_r^-\leq l)=
\sum_{q=1}^r\mathbf{P}(I_r^-\leq l\mid J_r=q)\mathbf{P}(J_r=q)
\geq
\frac{l}{r}\sum_{q=1}^r\mathbf{P}(J_r=q)
=
\frac{l}{r}.
$$

For $I_r^+$, by following the above analysis, we have
$$
\mathbf{P}(I_r^+\leq l\mid  J_r=q)
=
\begin{cases}
\displaystyle
\frac{\binom{l}{q}}{\binom{r}{q}},
& q\leq l,\\[2ex]
0,
& q>l.
\end{cases}
$$
It is enough to consider the case that $\{J_r=1\}$.  Choose $m_r\geq2$ such that
$p_{m_r}^*\leq\frac{1}{r}<p_{m_r-1}^*.$ Since $p_{m_r}^*\geq\eta p_{m_r-1}^*$, then
$\frac{\eta}{r}<p_{m_r}^*\leq\frac{1}{r}.$
The event that exactly one of $\zeta_1^*,\ldots,\zeta_r^*$ is at least $m_r$ implies $J_r=1$. Consequently,
$$
\mathbf{P}(J_r=1)
\geq
r p_{m_r}^*(1-p_{m_r}^*)^{r-1}
\geq
\eta\left(1-\frac{1}{r}\right)^{r-1}=\eta\left(1+\frac{1}{r-1}\right)^{-(r-1)}
\geq  \frac{\eta}{e}.
$$
By $\mathbf{P}(I_r^+\leq l\mid J_r=1)=\frac{l}{r},$ we obtain
$$
\mathbf{P}(I_r^+\leq l)
\geq
\mathbf{P}(I_r^+\leq l,\ J_r=1)
\geq \frac{\eta}{e}\frac{l}{r}.
$$
Thus,
$\mathbf{P}(I_r^\pm\leq l)\geq\frac{\eta}{e}\frac{l}{r},$
which together with (\ref{eq:3.41-b}) implies that  \eqref{eq:3.31} holds with
$c_1=\frac{\eta}{e},~c_2=\frac{2}{\eta}.$

Finally, fix $K\geq1$. For every integer $m\geq1$,  by Lemma \ref{lem:3.4} and (\ref{eq:3.21}), we have
$\sum_{k>K}\mathbf{P}(\zeta_k^*>m)=\infty.$
The second Borel-Cantelli lemma and a countable intersection over $m$ yield $\sup_{k>K}\zeta_k^*=\infty,~\text{a.s.}.$
On the other hand,
$$
\max_{1\leq k\leq K}\zeta_k^*<\infty,~~\text{a.s.},
$$
since for any integer $k\geq 1$, $\tau_k<\infty~ \text{a.s.}.$ Thus, almost surely, there exists $L>K$ such that
$\zeta_L^*>\max_{1\leq k\leq K}\zeta_k^*.$  It follows that for any $r\geq L$, the maximum among $\zeta_1^*,\ldots,\zeta_r^*$ cannot be attained within the first $K$ blocks.
And thus
$$
I_r^->K \quad\text{and}\quad I_r^+>K.
$$
Since $K$ is arbitrary, \eqref{eq:3.32} follows. The proof is complete.
\end{proof}

\medskip

\begin{proof}[\bf Proof of Lemma \ref{lem:3.7}]
In the following, we only give the proof of (\ref{lem-3.7-a}). The proof of (\ref{lem-3.7-b}) is similar.

We decompose
$$
\mathbf{P}(I_r^+\leq l,\ I_s^+\leq L)
=
\mathbf{P}(I_r^+\leq l,\ I_s^+\leq L,\ I_s^+\leq r)
+
\mathbf{P}(I_r^+\leq l,\ r<I_s^+\leq L).
$$
On $\{I_s^+\leq r\}$, we have $I_s^+=I_r^+.$  Then by Lemma~\ref{lem:3.6}, we have
\begin{equation}\label{eq:3.42}
\mathbf{P}(I_r^+\leq l,\ I_s^+\leq L,\ I_s^+\leq r)
\leq
\mathbf{P}(I_s^+\leq l)
\leq
C\frac{l}{s}.
\end{equation}

If $L\leq r$, then $\mathbf{P}(I_r^+\leq l,\ r<I_s^+\leq L)=0$. Suppose that $L>r$, and set
$$
\mathcal{J}_{r,s}
:=
\left\{1\leq j\leq s-r:\zeta_{r+j}^*=\max_{1\leq i\leq s-r}\zeta_{r+i}^*\right\},
~~
I_{r,s}^-:=\min\mathcal{J}_{r,s},
~~
I_{r,s}^+:=\max\mathcal{J}_{r,s}.
$$
The random variables $I_r^+$ and $I_{r,s}^+$ depend on disjoint families of i.i.d. random variables. Hence they are independent, and
$I_{r,s}^+\overset{\mathrm d}{=}I_{s-r}^+.$
Moreover,
$\{r<I_s^+\leq L\}\subseteq\{I_{r,s}^+\leq L-r\}.$
 By the assumption that $s\geq 2r$ and $1\leq L\leq s/2$ in Lemma 3.7, we get that
$1\leq L-r\leq \frac{s}{2}-r\leq \frac{s-r}{2}, s-r\geq \frac{s}{2}$.
Lemma \ref{lem:3.6} then yields
\begin{equation}\label{eq:3.43}
\mathbf{P}(I_r^+\leq l,\ r<I_s^+\leq L)
\leq
\mathbf{P}(I_r^+\leq l)\mathbf{P}(I_{r,s}^+\leq L-r)
\leq
C\frac{l}{r}\frac{L-r}{s-r}
\leq
C\frac{l}{r}\frac{L}{s},
\end{equation}
which together with \eqref{eq:3.42} implies that (\ref{lem-3.7-a}). The proof is complete.
\end{proof}

\medskip

\begin{proof}[\bf Proof of Lemma \ref{pro:3.8}]
Suppose first that
$$
\sum_{r=2}^{\infty}\frac{1}{rb(r)}<\infty.
$$
Then $b(r)\to\infty$. Fix $A>0$.  Define
$k_i:=\left\lceil\frac{i b(i)}{A}\right\rceil, \forall i\geq1.$
Let $i_0$ be a positive integer such that $b(i)>A$ for all $i\ge i_0$.
By \eqref{eq:3.30},
$$
\sum_{i=i_0}^{\infty}\mathbf{P}(I_{k_i}^\pm=i)
\leq
C\sum_{i=i_0}^{\infty}\frac{1}{k_i}
\leq
CA
\sum_{i=i_0}^{\infty}\frac{1}{i b(i)}<\infty.
$$
The Borel--Cantelli lemma gives
\begin{align}\label{proof-lem3.8-a}
\mathbf{P}(I_{k_i}^\pm=i\ \mathrm{i.o.})=0.
\end{align}

Suppose that $\frac{b(r)I_r^\pm}{r}\leq A$ for infinitely many $r$. Then there exists a strictly increasing sequence $(r_i)_{i\geq1}$ such that
$\frac{b(r_i)I_{r_i}^\pm}{r_i}\leq A,\ \forall i\geq1.$
By \eqref{eq:3.32}, $I_{r_i}^\pm\to\infty\  \text{a.s.}.$ Moreover, since $I_{r_i}^\pm\leq r_i$ and $b(\cdot)$ is non-decreasing,
$$
r_i
\geq
\frac{I_{r_i}^\pm b(r_i)}{A}
\geq
\frac{I_{r_i}^\pm b(I_{r_i}^\pm)}{A}\geq k_{I_{r_i}^\pm}~~\text{a.s.},
$$
which together with the fact that $b(I_{r_i}^\pm)\to\infty~\text{a.s.}$, implies that for all sufficiently large $i$, we have
$I_{r_i}^\pm\leq k_{I_{r_i}^\pm}\leq r_i \ \text{a.s.}.$ Then by the definitions in \eqref{3.4-1}, we get that
$I_{k_{I_{r_i}^\pm}}^\pm=I_{r_i}^\pm\  \text{a.s.}$
for all sufficiently large $i$. Since $I_{r_i}^\pm\to\infty \ \text{a.s.}$ (cf. (\ref{eq:3.32})),
this equality holds for infinitely many distinct values of $I_{r_i}^\pm$. Hence, up to a null event,
$$
\left\{\frac{b(r)I_r^\pm}{r}\leq A\ \mathrm{i.o.}\right\}
\subseteq
\left\{I_{k_j}^\pm=j\ \mathrm{i.o.}\right\},
$$
 which together with \eqref{proof-lem3.8-a} implies that
$\mathbf{P}\left(\frac{b(r)I_r^\pm}{r}\leq A\ \mathrm{i.o.}\right)=0.$
Taking $A\in\mathbb{N}$ gives\linebreak
$\liminf_{r\to\infty}\frac{b(r)I_r^\pm}{r}=\infty\ \text{a.s.}$.

Suppose next that
$$
\sum_{r=2}^{\infty}\frac{1}{rb(r)}=\infty.
$$
Since $(rb(r))^{-1}$ is non-increasing, the Cauchy condensation theorem \cite[Theorem~3.27]{Rudin1976} gives
$\sum_{q=1}^{\infty}\frac{1}{b(2^q)}=\infty.$
Fix $0<\delta<b(1)/2$, and set
$$
\ell_q:=\left\lfloor\frac{\delta 2^q}{b(2^q)}\right\rfloor,
~~~
\mathcal{Q}:=\{q\geq1:\ell_q\geq1\}.
$$
For every $q\in\mathcal{Q}$, $1\leq\ell_q<2^{q-1}.$ If $q\notin\mathcal{Q}$, then
$\delta 2^q/b(2^q)<1$, and hence
$\sum_{q\notin\mathcal{Q}}\frac{1}{b(2^q)}
\leq\frac{1}{\delta}\sum_{q=1}^{\infty}\frac{1}{2^q}<\infty.$
It follows that
$\sum_{q\in\mathcal{Q}}\frac{1}{b(2^q)}=\infty.$
Moreover, for every $q\in\mathcal{Q}$,
$\ell_q\geq\frac{\delta 2^q}{2b(2^q)},$
and therefore
$$
\sum_{q\in\mathcal{Q}}\frac{\ell_q}{2^q}
\geq
\frac{\delta}{2}\sum_{q\in\mathcal{Q}}\frac{1}{b(2^q)}=\infty.
$$

For $u\geq1$, set
$$
\mathcal{Q}_u:=\mathcal{Q}\cap\{1,\ldots,u\},
~~
Q_u^\pm:=\sum_{q\in\mathcal{Q}_u}\mathbf{1}_{\{I_{2^q}^\pm\leq\ell_q\}}.
$$
Lemma~\ref{lem:3.6} gives
\begin{equation}\label{eq:3.44}
c\sum_{q\in\mathcal{Q}_u}\frac{\ell_q}{2^q}
\leq
\mathbf{E}[Q_u^\pm]=\sum_{q\in\mathcal{Q}_u}\mathbf{P}(I_{2^q}^\pm\leq\ell_q)
\leq
C\sum_{q\in\mathcal{Q}_u}\frac{\ell_q}{2^q},
\end{equation}
and hence
$\mathbf{E}[Q_u^\pm]\to\infty.$ For $q<h$ in $\mathcal{Q}$, $2^h\geq2^{q+1}$, and
Lemma~\ref{lem:3.7} gives
\begin{equation}\label{eq:3.45}
\mathbf{P}
\left(I_{2^q}^\pm\leq\ell_q,\,I_{2^h}^\pm\leq\ell_h\right)
\leq
C\frac{\ell_q}{2^h}+C\frac{\ell_q}{2^q}\frac{\ell_h}{2^h}
=
C\frac{\ell_q}{2^q}2^{q-h}+C\frac{\ell_q}{2^q}\frac{\ell_h}{2^h}.
\end{equation}
By \eqref{eq:3.44}, \eqref{eq:3.45}, and $\sum_{h>q}2^{q-h}\leq1$, we obtain
$$
\begin{aligned}
\mathbf{E}[(Q_u^\pm)^2]
&=
\sum_{q\in\mathcal{Q}_u}
\mathbf{P}(I_{2^q}^\pm\leq\ell_q)+2\sum_{\substack{q,h\in\mathcal{Q}_u\\q<h}}
\mathbf{P}\left(I_{2^q}^\pm\leq\ell_q,\,I_{2^h}^\pm\leq\ell_h\right)\\
&\leq
C\sum_{q\in\mathcal{Q}_u}
\frac{\ell_q}{2^q}+C\sum_{\substack{q,h\in\mathcal{Q}_u\\q<h}}\frac{\ell_q}{2^q}2^{q-h}
+
C\left(\sum_{q\in\mathcal{Q}_u}\frac{\ell_q}{2^q}\right)^2\\
&\leq
C\left(\mathbf{E}[Q_u^\pm]+\left(\mathbf{E} [Q_u^\pm]\right)^2\right)
\leq  C\left(\mathbf{E}[Q_u^\pm]\vee 1\right)^2,
\end{aligned}
$$
which together with  $\mathbf{E}[Q_u^\pm]\to\infty$, implies that
$\mathbf{E}[(Q_u^\pm)^2]\leq C \left(\mathbf{E}[Q_u^\pm]\right)^2$
for all sufficiently large $u$. The Paley-Zygmund lemma \cite[Lemma~5.1]{K2021} yields for all sufficiently large $u$,
$$
\mathbf{P}\left(Q_u^\pm\geq\frac{1}{2}\mathbf{E}[Q_u^\pm]\right)
\geq
\left(1-\frac{1}{2}\right)^2 \frac{\left(\mathbf{E}[Q_u^\pm]\right)^2}{\mathbf{E}[(Q_u^\pm)^2]}
\geq
\frac{1}{4C}=:c>0.
$$
Since $\mathbf{E}[Q_u^\pm]\to\infty$, for any $M>0$, we can choose $u_M$ such that $\frac12\mathbf{E}[Q_u^\pm]\ge M$. Then $\mathbf{P}\left(\sup_{u\geq1}Q_u^\pm\ge M\right)\ge c$,
and continuity from above gives $\mathbf{P}\left(\sup_{u\geq1}Q_u^\pm=\infty\right)\geq c>0.$
Consequently,
$\mathbf{P}(I_{2^q}^\pm\leq\ell_q\ \mathrm{i.o.})
=
\mathbf{P}\left(\sup_{u\geq1}Q_u^\pm=\infty\right)\geq c>0.$
On the event $\{I_{2^q}^\pm\leq\ell_q\}$, the definition of $\ell_q$ gives
$\frac{b(2^q)I_{2^q}^\pm}{2^q}\leq\frac{b(2^q)\ell_q}{2^q}\leq\delta.$
Then
$\left\{I_{2^q}^\pm\leq\ell_q\ \mathrm{i.o.}\right\}
\subseteq
\left\{\liminf_{r\to\infty}\frac{b(r)I_r^\pm}{r}\leq\delta\right\}.$
Therefore,
$$
\mathbf{P}\left(\liminf_{r\to\infty}\frac{b(r)I_r^\pm}{r}\leq\delta\right)
\geq
\mathbf{P}\left(I_{2^q}^\pm\leq\ell_q\ \mathrm{i.o.}\right)
\geq c>0.
$$

For $i\geq1$ and $r>i$, define
$$
\bar{\mathcal{J}}_r^i:=\left\{i<k\leq r:\zeta_k^*=\max_{i<j\leq r}\zeta_j^*\right\},
~~
\bar{I}_r^{-,i}:=\min\bar{\mathcal{J}}_r^i,
~~
\bar{I}_r^{+,i}:=\max\bar{\mathcal{J}}_r^i.
$$
By \eqref{eq:3.32}, for every fixed $i$, $I_r^\pm=\bar{I}_r^{\pm,i}$ a.s. for all sufficiently large $r$.
 Hence
\begin{equation}\label{eq:3.46}
\mathbf{P}\left(\left\{\liminf_{r\to\infty}\frac{b(r)I_r^\pm}{r}\leq\delta\right\}
\mathbin{\triangle}
\left\{\liminf_{r\to\infty}\frac{b(r)\bar{I}_r^{\pm,i}}{r}\leq\delta\right\}\right)
=0.
\end{equation}
The second event in \eqref{eq:3.46} belongs to $\sigma(\zeta_k^*:k>i).$ Since $i$ is arbitrary, the first event belongs to the completed tail $\sigma$-field of $(\zeta_k^*)_{k\geq1}$. By Lemma~\ref{lem:3.4}, Kolmogorov's 0-1 law, and the preceding lower bound, we obtain
$$
\mathbf{P}\left(\liminf_{r\to\infty}\frac{b(r)I_r^\pm}{r}\leq\delta\right)=1.
$$
Taking $\delta=b(1)/(4j)$, $j\geq1$, yields
$\liminf_{r\to\infty}\frac{b(r)I_r^\pm}{r}=0\ \text{a.s.}.$
This proves \eqref{eq:3.33}.  The proof is complete.
\end{proof}

\bigskip

 \noindent {\bf\large Acknowledgments}\  This work was supported by  National Key R\&D Program of China (Grant No. 2023YFA1009600),  National Natural Science Foundation of China (Grant Nos. 12171335, 12471139, 12301603) and  Simons Foundation (\#429343).


\begin{thebibliography}{1234}

\bibitem[Bass (2023)]{B2023}
Bass, R. F., The rate of escape of the most visited site of Brownian motion.
\textit{Electron. J. Probab.} \textbf{28}(2023), 1-12.

\bibitem[Bass--Eisenbaum--Shi (2000)]{BES2000}
Bass, R. F., Eisenbaum, N., Shi, Z., The most visited sites of symmetric stable processes.
\textit{Probab. Theory Relat. Fields} \textbf{116}(2000), 391-404.

\bibitem[Bass--Griffin (1985)]{BG1985}
Bass, R. F., Griffin, P. S., The most visited site of Brownian motion and simple
random walk.
\textit{Z. Wahrsch. Verw. Gebiete} \textbf{70}(1985), 417-436.

\bibitem[Bertoin--Marsalle (1998)]{BM1998}
Bertoin, J., Marsalle, L., Point le plus visit\'{e} par un mouvement brownien avec
d\'{e}rive.
In \textit{S\'eminaire de Probabilit\'es XXXII}, \textit{Lecture Notes in Math.}
\textbf{1686}, Springer, Berlin, 1998, 397-411.

\bibitem[Bini--Latouche--Meini (2005)]{BM2005}
Bini, D.~A., Latouche, G., Meini, B.,
\textit{Numerical Methods for Structured Markov Chains}.
Oxford University Press, Oxford, 2005.


\bibitem[Cs\'aki et al. (2005)]{CFRRS2005}
Cs\'aki, E., F\"oldes, A., R\'ev\'esz, P., Rosen, J., Shi, Z.,
Frequently visited sets for random walks.
\textit{Stochastic Process. Appl.} \textbf{115} (2005), 1503-1517.


\bibitem[Ding--Shen (2018)]{DS2018}
Ding, J., Shen, J.,
Three favorite sites occurs infinitely often for one-dimensional simple random walk.
\textit{Ann. Probab.} \textbf{46}(2018), 2545-2561.



\bibitem[Durrett (2019)]{D2019}
Durrett, R., \textit{Probability: Theory and Examples (5th ed.)}, 
Cambridge University Press, Cambridge, 2019.


\bibitem[Erd\H{o}s--R\'ev\'esz (1984)]{ER1984}
Erd\H{o}s, P., R\'ev\'esz, P.,
On the favorite points of a random walk.
In \textit{Mathematical Structures Computational Mathematics Mathematical Modelling}, vol.~2, Bulgarian Academy of Sciences, Sofia, 1984, 152-157.

\bibitem[Erd\H{o}s--R\'ev\'esz (1987)]{ER1987}
Erd\H{o}s, P., R\'ev\'esz, P.,
Problems and results on random walks.
In \textit{Mathematical Statistics and Probability Theory}, vol.~B, Reidel, Dordrecht, 1987, 59-65.

\bibitem[Erd\H{o}s--R\'ev\'esz (1991)]{ER1991}
Erd\H{o}s, P., R\'ev\'esz, P.,
Three problems on the random walk in $\mathbb{Z}^d$.
\textit{Studia Sci. Math. Hungar.} \textbf{26}(1991), 309-320.


\bibitem{Gut2009}
Gut, A.,
\textit{Stopped Random Walks: Limit Theorems and Applications},
Applied Probability, Vol.~8, Springer, New York, 2009.


\bibitem[Hao-Hu-Ma-Song (2026)]{HHMS2026}
Hao, C. X., Hu, Z. C., Ma, T., Song, R.,
Three favorite edges occurs infinitely often for one-dimensional simple random walk.
\textit{Commun. Math. Stat.} \textbf{14}(2026), 391-414.

\bibitem[Hao-Li-Okada-Zheng (2025)]{HLOZ2025}
Hao, C. X., Li, X. Y., Okada, I., Zheng, Y. S.,
Favorite sites for simple random walk in two and more dimensions.
\textit{Probab. Theory Relat. Fields}  \textbf{195}(2026), 1765-1822.

\bibitem[Hu--Shi (2000)]{HS2000}
Hu, Y., Shi, Z.,
The problem of the most visited site in random environment.
\textit{Probab. Theory Relat. Fields} \textbf{116}(2000), 273-302.

\bibitem[Kallenberg (2021)]{K2021}
Kallenberg, O.,
\textit{Foundations of Modern Probability (3rd ed.)}, Probability Theory and Stochastic Modelling, Vol.~99, Springer,  2021.



\bibitem[Lifshits--Shi (2004)]{LS2004}
Lifshits, M. A., Shi, Z.,
The escape rate of favorite sites of simple random walk and Brownian motion.
\textit{Ann. Probab.} \textbf{32}(2004), 129-152.


\bibitem[Marcus-Rosen (2006)]{MR2006}
Marcus, M. B., Rosen, J.,
\textit{Markov Processes, Gaussian Processes, and Local Times}.
Cambridge Univ. Press, Cambridge, 2006.

\bibitem{Oka16}
Okada, I.,
Topics and problems on favorite sites of random walks.
\textit{Stochastic analysis on large scale interacting systems}, 129-139, RIMS K\^{o}ky\^{u}roku Bessatsu, B 59, Res. Inst. Math. Sci. (RIMS), Kyoto, 2016.

\bibitem[Rudin (1976)]{Rudin1976}
Rudin, W.,
\textit{Principles of Mathematical Analysis (3rd ed.)}, McGraw--Hill, New York, 1976.

\bibitem[Shi--T\'oth (2000)]{ST2000}
Shi, Z., T\'oth, B.,
Favourite sites of simple random walk.
\textit{Period. Math. Hungar.} \textbf{41}(2000), 237-249.

\bibitem[Spitzer (2001)]{S2001}
Spitzer, F., \textit{Principles of Random Walk (2nd ed.)},   Springer, 2001.

\bibitem[Sznitman--Zerner (1999)]{SZ1999}
Sznitman, A. S., Zerner, M. P. W.,
A law of large numbers for random walks in random environment.
\textit{Ann. Probab.} \textbf{27}(1999), 1851-1869.

\bibitem[Vilkas (2025)]{Vilkas2025}
Vilkas, T.,
Left-continuous random walk on $\mathbb{Z}$ and the parity of its hitting times.
\textit{J. Theoret. Probab.} \textbf{38}(2025), Paper No.~68.

\bibitem[T\'oth (2001)]{T2001}
T\'oth, B.,
No more than three favorite sites for simple random walk.
\textit{Ann. Probab.} \textbf{29}(2001), 484-503.

\bibitem[T\'oth--Werner (1997)]{TW1997}
T\'oth, B., Werner, W.,
Tied favourite edges for simple random walk.
\textit{Combin. Probab. Comput.} \textbf{6}(1997), 359-369.

\bibitem[Zhou--Hu--Song (2025)]{ZHS2025}
Zhou, G. S., Hu, Z. C., Song, R.,
Favorite sites of one-dimensional asymmetric simple random walk.
arXiv: 2511.02316, 2025.


\end{thebibliography}
\end{document}